\documentclass{amsart}
\usepackage[utf8]{inputenc}
\usepackage[dvipsnames]{xcolor}
\usepackage{amsmath}
\usepackage{amsthm}
\usepackage{thmtools}
\usepackage{amssymb}
\usepackage{hyperref}
\usepackage{wasysym}
\usepackage{tikz}
\usepackage{tikz-cd}
\usetikzlibrary{arrows.meta, decorations.markings}
\usepackage{pdfpages}
\usepackage{amssymb}
\usepackage{mathtools}
\usepackage{enumitem}
\usepackage{mathrsfs}
\usepackage[all]{xy}

\usepackage{mathtools}
\usepackage{hyperref}
\usepackage{cleveref}
\usepackage{url}
\usepackage{bbm}
\usepackage{booktabs}
\usepackage{comment}

\usepackage[T1]{fontenc}
\usepackage[british]{babel}

\usepackage[margin=1.3in]{geometry}
\numberwithin{equation}{section} \numberwithin{figure}{section}

\DeclareMathOperator{\Gal}{Gal} 

\DeclareMathOperator{\id}{id}
\DeclareMathOperator{\Spec}{Spec}

\DeclareMathOperator{\SL}{SL}
\DeclareMathOperator{\PSL}{PSL}
\DeclareMathOperator{\Frob}{Frob}

\DeclareMathOperator{\Schemes}{\underline{Schemes}}
\DeclareMathOperator{\Sets}{\underline{Sets}}

\newcommand\FF{\mathbb{F}}
\newcommand\PP{\mathbb{P}}
\newcommand\ZZ{\mathbb{Z}}

\newcommand\QQ{\mathbb{Q}}
\newcommand\RR{\mathbb{R}}
\newcommand\CC{\mathbb{C}}

\newcommand{\cH}{\mathcal{H}}
\newcommand{\cU}{\mathcal{U}}

\newtheoremstyle{one}
{11pt}% measure of space to leave above the theorem. E.g.: 3pt
{11pt}% measure of space to leave below the theorem. E.g.: 3pt
{\slshape}% name of font to use in the body of the theorem
{}% measure of space to indent
{\sc}% name of head font
{.}% punctuation between head and body
{1mm}% space after theorem head
{}% Manually specify head

\theoremstyle{one} 
\newtheorem{theorem}{\textbf{Theorem}}[section]
\newtheorem{proposition}[theorem]{\textbf{Proposition}}
\newtheorem{lemma}[theorem]{\textbf{Lemma}}
\newtheorem{corollary}[theorem]{\textbf{Corollary}}

\theoremstyle{definition}

\newtheorem{remark}[theorem]{Remark}
\newtheorem*{notation}{Notation and conventions}
\newtheorem*{outline}{Outline of the paper}


\title{The Hasse norm principle for $A_4$-quartic extensions of global function fields}
\author{Anand Deopurkar, Rachel Newton, Vaidehee Thatte, Rosa Winter}
\date{}

\begin{document}

\begin{abstract}
For a finite extension of global fields $K/k$, the norm map $N_{K/k}: K^{\times} \rightarrow k^{\times}$ extends to a map on idèle groups. The Hasse norm principle holds if every element of $k^\times$ that is a norm everywhere locally is also a norm globally. 
In this paper, we study the statistics of the Hasse norm principle in a setting that is out of reach in the number field context, namely that of $A_4$-quartic extensions. 
We show that failures of the Hasse norm principle are generally rare for \(A_4\)-quartic extensions of global function fields \(\mathbb{F}_q(t)\).
We achieve this by introducing a decorated Hurwitz space parametrising the failures of the Hasse norm principle and then using the Chebotarev density theorem to estimate their frequency.
\end{abstract}

\maketitle

\section{Introduction}\label{sec: Introduction}

Let $K/k$ be a finite extension of global fields.
Denote by \(N_{K/k} \colon K^{\times} \rightarrow k^{\times}\) the norm map and recall that it extends to a map on id\`{e}le groups \(N_{K/k} \colon \mathbb{A}_K^{\times} \to \mathbb{A}_k^{\times}\).
The \emph{Hasse norm principle} is said to hold for $K/k$ if, for an element of $k^\times$, being a global norm from $K$ is equivalent to being a norm everywhere locally. In other words, the Hasse norm principle holds if the so-called \emph{knot group}, defined as $$\mathfrak{K}(K/k)=(k^\times\cap N_{K/k}\mathbb{A}_K^\times)/N_{K/k}K^\times,$$ is trivial. Historically, most of the work on the Hasse norm principle has been presented in the context of number fields, although many of the results also extend to the setting of global fields of arbitrary characteristic. For example, Hasse's original norm theorem states that the Hasse norm principle holds for cyclic extensions of number fields, but it is also valid for cyclic extensions of global function fields. 

Beyond the cyclic case, the Hasse norm principle has been proven to hold, for example, when $[K:k]=p$ for $p$ prime \cite{bar:81}; when $[K:k]=2p$ for $p$ a prime with $2p-1$ not a prime power~\cite{ost:}; and when $[K:k]=n$ with Galois group either $D_n$ \cite{bar:81*1}, or $S_n$ \cite{kun.vos:84} (see also \cite{vos:88}), or $A_n$ with $n\geq5$ \cite{mac:20}. Again, while they may be stated for number fields, these results hold for extensions of global fields of arbitrary characteristic --- see the references in the proof of \Cref{lem:MN}.

On the other hand, Hasse already showed that not all extensions satisfy the Hasse norm principle by demonstrating that in the abelian extension $\mathbb{Q}(\sqrt{-3},\sqrt{13})/\QQ$ the number $3$ gives a counterexample \cite{has:31}. In \cite{bro.new:16}, Browning and the second author calculated the frequency of counterexamples
to the Hasse norm principle for a finite extension $K/\mathbb{Q}$ in terms of the cardinality of the knot group, which they computed for bicyclic extensions. In a similar vein, the authors of \cite{man.new.ozm.ea:21} computed the proportion of polynomials of degree $d$ in $\mathbb{F}_q[t]$ that are everywhere locally norms but fail to be global
norms from a finite extension $K/\mathbb{F}_q(t)$.

The Hasse norm principle has also been studied in families of number fields. Together with Frei and Loughran, the second author showed in \cite{fre.lou.new:18} that for every non-cyclic finite abelian group $G$ and number field $k$, there exists a $G$-extension of $k$ for which the Hasse norm principle fails. In the same paper, the authors showed that, when ordered by discriminant, the Hasse norm principle fails for a positive proportion of finite abelian $G$-extensions if and only if $G/G[p]$ is not cyclic, where $p$ is the smallest prime dividing $|G|$. These results were strengthened by Koymans and Rome in~\cite{koy.rom:24b}. Asymptotics for the number of biquadratic extensions of $\mathbb{Q}$ for which the Hasse norm principle fails were given by Rome \cite{rom:18}. Similar asymptotics were given for a larger class of extensions by Koymans and Rome in~\cite{koy.rom:24a}. 
Finally, Frei, Loughran and the second author proved that, for a fixed number field $k$ and a fixed abelian group $G$, the Hasse norm principle holds for 100\% of extensions of $k$ with Galois group $G$, when ordered by conductor~\cite{fre.lou.new:22}.

Moving beyond the abelian setting is challenging because results on counting number fields, even without imposing the local conditions needed to guarantee validity (or failure) of the Hasse norm principle, are few and far between. Moreover, in some non-abelian families where counting results are available, such as cubic~\cite{dav.heil:71, dat.wri:88, sha.tho:24} or sextic~\cite{bha.woo.08} extensions with Galois group $S_3$, the Hasse norm principle is known to hold for all extensions in the family, so there is no statistical analysis to be done. 
The only studies of the statistics of the Hasse norm principle in families of non-abelian extensions of number fields of which we are aware are those carried out by Macedo for $D_4$-octics~\cite{macedoPhD:21} and by the second author and Varma for extensions with Galois group $S_4$ or $S_5$~\cite{nv:}. 

In this paper, we demonstrate the advantages of the global function field setting by studying the statistics of the Hasse norm principle in a setting that is currently out of reach in the number field context, namely that of $A_4$-quartic extensions.
By \emph{an \(A_4\)-quartic extension}, we mean a degree 4 extension with Galois group $A_4$. Our main result is the following. 

  \begin{theorem}\label{thm: main result}
    Let \(q\) be a power of a prime \(p \geq 5\).
    Given \(\epsilon > 0\), if \(n\) is sufficiently large, then
    \[  \limsup_{m \to \infty} \frac{\left|\left\{\parbox{.62\textwidth}{Isomorphism classes of \(A_4\)-quartic extensions of \(\mathbb{F}_{q^m}(t)\) with discriminant of degree \(2n\) that fail the Hasse norm principle}\right\}\right|}{\left|\left\{\parbox{.62\textwidth}{Isomorphism classes of \(A_4\)-quartic extensions of \(\mathbb{F}_{q^m}(t)\) with discriminant of degree \(2n\)}\right\}\right|} \leq (4+\epsilon) n^{-1/3}.\]
  \end{theorem}
  \begin{remark}
    \begin{enumerate}
    \item 
    For any field \(k\), a finite extension of \(k(t)\) is equivalent to a finite cover \(f \colon C \to \PP^1_k\), where \(C\) is a smooth projective curve over \(k\).
      In our case, the degree of \(f\) is \(4\) and the characteristic of \(k\) is at least \(5\), so \(f\) is separable.
      By the \emph{discriminant} of the extension, we mean the branch divisor of \(f\).
      For \(A_4\)-quartics, the branch divisor is even.
      In \cite{ell.ven:05}, ``discriminant'' is used slightly differently: there, it means \(q\) raised to the degree of the ramification divisor. (Note that the degree of the ramification divisor is the same as the degree of the branch divisor.)
            
    \item
    In Theorem \ref{thm: main result} we require $n$ to be sufficiently large. This comes from
    Proposition~\ref{prop:cpa4} and Lemma~\ref{lem:estimates}; see Remark \ref{rem:suff large}.
    \end{enumerate}
  \end{remark}

We obtain the following immediate corollary.
\begin{corollary}\label{Cor: main result}
Let $q$ be a power of a prime $p\geq5$.
     Then we have
    \[ \lim_{n \to \infty} \limsup_{m \to \infty} \frac{\left|\left\{\parbox{.62\textwidth}{Isomorphism classes of \(A_4\)-quartic extensions of \(\mathbb{F}_{q^m}(t)\) with discriminant of degree \(2n\) that fail the Hasse norm principle}\right\}\right|}{\left|\left\{\parbox{.62\textwidth}{Isomorphism classes of \(A_4\)-quartic extensions of \(\mathbb{F}_{q^m}(t)\) with discriminant of degree \(2n\)}\right\}\right|}  = 0.\]
\end{corollary}

Note that if one replaces the base field $\mathbb{F}_{q^m}(t)$ in Theorem \ref{thm: main result} and Corollary \ref{Cor: main result} with a number field $k$, then even the total number of $A_4$-quartic extensions of $k$ appearing in the denominator is not known: it is a famous open case of Malle's conjecture on the asymptotic distribution of $G$-extensions of bounded discriminant~\cite{mal:02, mal:04}. For $A_4$-quartic extensions of $\QQ$, the lower bound predicted by Malle has recently been proven by Loughran and Paterson~\cite{lou.pat:25} but the upper bound remains out of reach --- see~\cite{bstttz:20} for the best result to date.

An $A_4$-quartic extension of $\mathbb{F}_{q}(t)$ with discriminant of degree $2n$ corresponds to a degree $4$ cover of curves $f \colon C\to\PP^1_{\FF_q}$ with Galois group $A_4$ and branch divisor of degree $2n$.  Such covers are parametrised by a Hurwitz space which we denote by $H(n)$ (see~\Cref{sec:hurwitz} and~\ref{eq:disjoint} for its definition). Ellenberg and Venkatesh~\cite{ell.ven:05} were the first to use the topology of Hurwitz spaces over $\CC$ to control their point counts over finite fields and thus obtain heuristics for the number of finite (geometrically connected) $G$-extensions of $\FF_q(t)$, initiating a great deal of activity and progress in this area. T\"{u}rkelli~\cite{trk:15} extended the work of Ellenberg and Venkatesh to tackle covers of $\PP^1_{\FF_q}$ that are not necessarily geometrically connected, and showed that the resulting heuristic suggested a modification of Malle's conjecture that avoided the counterexamples due to Kl\"{u}ners~\cite{kl:05} that were known at the time. Interestingly, despite counterexamples to T\"{u}rkelli's modification of Malle's conjecture being found by Wang in the number field setting~\cite{wan:}, it has now been proven for finite $G$-extensions of $\FF_q(t)$, with $q$ sufficiently large and coprime to $|G|$. Landesman and Levy~\cite{lan.lev:25-1} proved that the order of magnitude of the counting function for such $G$-extensions is as predicted by Türkelli, improving on an upper bound given by Ellenberg, Tran and Westerland~\cite{ell.tra.wes:}. An asymptotic formula with an explicit description of the leading constant was given by Santens~\cite{san:}.

In order to study the statistics of the Hasse norm principle in families of $G$-extensions of $\FF_q(t)$, one needs to count extensions with local conditions imposed. Ramification conditions are readily captured by the Hurwitz space machinery. Imposing conditions on residue degrees is more challenging as this information is lost in the geometric setting when one base changes to $\overline{\FF}_q$. The first example of counting extensions of global function fields with a residue degree condition was in \cite{liu.woo.zur:24} where the authors required that their extensions be completely split at all places over infinity. 

In this paper, we translate the conditions for failure of the Hasse norm principle into prescriptions for the ramification indices and residue degrees of our $A_4$-quartics. We define a decorated Hurwitz space $H^{\rm fail}(n)$ admitting a finite map $\phi: H^{\rm fail}(n) \to H(n)$, and show that the Hasse norm principle fails for an $A_4$-quartic cover if and only if it lies in the image of $\phi$. Thus, we are led to consider the quotient 
\[\frac{\left|\phi(H^{\rm fail}(n)(\FF_q))\right|}{\left|H(n)(\FF_q)\right|}.\]
We have an obvious upper bound \[ \frac{\left|\phi(H^{\rm fail}(n)(\FF_q))\right|}{\left|H(n)(\FF_q)\right|} \leq \frac{\left|H^{\rm fail}(n)(\FF_q)\right|}{\left|H(n)(\FF_q)\right|}\]
and one can understand the right-hand side using existing technology for point counts on
Hurwitz spaces. But this upper bound turns out to be too crude for our purposes, giving a positive constant in place of the asymptotically zero result of~\Cref{Cor: main result}. The key point is that the size of $\phi(H^{\rm fail}(n)(\FF_q))$ is vastly smaller than the size of $H^{\rm fail}(n)(\FF_q)$; the
$\FF_q$-points of $H^{\rm fail}(n)$ tend to cluster together in fibres over the $\FF_q$-points of $H(n)$. The Chebotarev density theorem allows us to get a handle on this clustering and thus prove our main result, \Cref{thm: main result}. 

We expect the techniques introduced in this paper to work in any setting where one can translate the failure of the Hasse norm principle into conditions on ramification indices and residue degrees.  In these cases, one can parametrise the failures by a 
Hurwitz space analogous to our fail space $H^{\rm fail}(n)$. For example, this is the case for biquadratic extensions, since analogues of Lemmas \ref{lem:MN} and \ref{lem: splitting conditions} hold by a result of Tate, see~\cite[p.198]{cas.fro:67}. However, biquadratic extensions can also be tackled by more elementary methods. In \cite[Proposition 1.8]{mac.new:22}, local conditions for the validity of the Hasse norm principle are given for extensions with Galois group $S_4,S_5,A_4$ or $A_5$. 
Although \cite[Proposition 1.8]{mac.new:22} is stated for number fields, it also holds for extensions of global function fields, see the proof of Lemma \ref{lem:MN}. For example, for Galois group $A_4$, the only extensions for which the Hasse norm principle can fail are those of degrees $4, 6$ and $12$. We expect the methods showcased here for $A_4$-quartic extensions to work for $A_4$-extensions of degrees $6$ and $12$ as well.

\begin{outline}The paper is organised as follows. Section \ref{Sec: preliminaries} contains the necessary background on ramification theory and the Chebotarev density theorem for function fields. We determine conditions for the validity of the Hasse norm principle for an $A_4$-quartic extension of $\mathbb{F}_q(t)$ in terms of ramification indices and residue degrees (Lemma \ref{lem: splitting conditions}). We use this in Section~\ref{sec:hurwitz}, where we introduce a fail space $H^{\rm{fail}}(n_1,n_2)$ as a modification of the Hurwitz space $H(n_1,n_2)$. These spaces depend on parameters $n_1$ and $n_2$ corresponding to different types of ramification. We prove that the $A_4$-quartics for which the Hasse norm principle fails can be parametrised by the image of a forgetful map $H^{\rm{fail}}(n_1,n_2)(\mathbb{F}_q)\to H(n_1,n_2)(\mathbb{F}_q)$ (Theorem \ref{thm:failset}). In Section \ref{sec:morehurwitz}, we show how to estimate the ratio of the size of this image to the size of the codomain $H(n_1,n_2)(\mathbb{F}_q)$ using the Chebotarev density theorem. We analyse connected components of our various Hurwitz spaces in Section \ref{sec:comphurwitz}, allowing us to compute this ratio in the large power of $q$ limit (Theorem \ref{thm:failbound}). Finally, in Section~\ref{sec: n to infinity} we analyse the limit as $n=n_1+n_2$ goes to infinity, proving \Cref{thm: main result}. 
\end{outline}

\begin{notation} 
  For a field extension $K/k$ and a place $v'$ of $ K$ above a place $v$ of $k$, we denote by $e(v'/v)$ and $f(v'/v)$ the ramification index and inertia degree, respectively. When $K$ and $k$ are function fields, we write $e(p'/p)$ and $f(p'/p)$ for the ramification index and residue degree, respectively, of the closed point $p'$ corresponding to $v'$ lying over the closed point $p$ corresponding to $v$.

  By \(\Schemes\) we mean the category of locally Noetherian schemes over \(\ZZ[1/6]\).
A \emph{scheme} means an object in this category.
If \(S\) is a scheme, then an \(S\)-scheme is a scheme together with a morphism to \(S\).
If \(X\) is an \(S\)-scheme and \(T \to S\) is a morphism, we denote by \(X_T\) the \(T\)-scheme \(X \times_S T \to T\).

A \emph{degree \(d\) cover} means a finite flat morphism of degree \(d\).

\end{notation}

\subsection*{Acknowledgements}
We are grateful to Jordan Ellenberg, Aaron Landesman, Scott Mullane and Melanie Matchett Wood for helpful discussions and to Harmeet Singh for pointing out some useful references.
This research was supported by UKRI Future Leaders Fellowship MR/T041609/1, MR/T041609/2 and UKRI1060.
Additionally, the first author was supported by the Australian Research Council grant DE180101360 and the last author by ERC Horizon 2020 Consolidator Grant ID 101001051 and MSCA Postdoctoral Fellowship 101148712. 

\begin{tabular}{c p{10cm}}
\raisebox{-25pt}{ \includegraphics[scale=0.3]{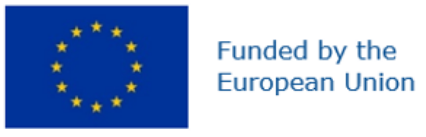}   }&   \tiny{ Funded by the European Union. Views and opinions expressed are however those of the author(s) only
and do not necessarily reflect those of the European Union or the European Research Executive Agency. Neither
the European Union nor the granting authority can be held responsible for them.}
\end{tabular}

\section{Preliminaries}\label{Sec: preliminaries}

This section gathers some necessary background and preliminary results that will eventually allow us to encode failures of the Hasse norm principle via Hurwitz spaces, and to measure the frequency of such failures.

\subsection{Local conditions}\label{sec: splitting primes}
In this section we classify the $A_4$-quartic extensions of $\mathbb{F}_q(t)$ for which the Hasse norm principle holds, in terms of the splitting of primes. We first set up notation and background results. 
\vspace{11pt}

Let $p$ be a prime, $q$ a power of $p$, and $K$ a finite extension of $k=\mathbb{F}_q(t)$ of degree coprime to $p$, with Galois closure $L$. Let $G$ be the Galois group Gal$(L/k)$. We assume $(p,|G|)=1$, and that $L/k$ does not contain any non-trivial constant sub-extensions. Let $H= \Gal(L/K)$.

For a place $v$ of $k$, and a place $w$ of $L$ above $v$, we denote by $k_v$ and $L_w$ the completions of $k$ and $L$ with respect to $v$ and $w$, respectively, and by $\mathbb{F}_v$ and $\mathbb{F}_w$ the corresponding residue fields. We write $D_{w/v}\subset G$ for the decomposition group of $w$ in $G$, that is, $D_{w/v}=\{g\in G\colon g(w)=w\}$. This is the Galois group of $L_w$ over $k_v$ \cite[Prop. II.9.6]{neu:99}.  The following diagram summarises what we have defined so far.

\begin{center}
\begin{tikzcd}
\tikzcdset{every label/.append style = {font = \normalsize}}
&\mbox{field extensions}&\mbox{places}&\mbox{completions}&\mbox{residue fields}\\
&L\arrow[d, dash, bend left=50,"H"]&w&L_w\arrow[d, dash, bend left=50,"D_{w/v}\cap H"]&\mathbb{F}_w\\
&K \arrow[u, dash]&v' \arrow[u, dash]&K_{v'} \arrow[u, dash]&\mathbb{F}_{v'}\arrow[u, dash]\\
&k=\mathbb{F}_q(t) \arrow[u, dash] \arrow[uu, dash, bend left=50,"G"] &v \arrow[u, dash]&k_v \arrow[u, dash] \arrow[uu,dash, bend left=50,"D_{w/v}"]&\mathbb{F}_v \arrow[u, dash]\\
&&&&\mathbb{F}_q \arrow[u, dash]
\end{tikzcd}
\end{center}

 We denote by $I_{w/v}$ the inertia group of $w$, that is, the subgroup of $D_{w/v}$ that acts trivially on the residue field $\mathbb{F}_w$. We have a short exact sequence \cite[Prop. II.9.9]{neu:99}: \begin{equation}\label{ses Dv Iv}
0\longrightarrow I_{w/v}\longrightarrow D_{w/v}\longrightarrow \mbox{Gal}(\mathbb{F}_w/\mathbb{F}_v)\longrightarrow 0.
\end{equation}

We write $D_v$ to denote a choice of $D_{w/v}$ for some place $w$ of $L$ above a place $v$ of $k$, and we use the notation $I_v$ similarly. The groups $D_v$ and $I_v$ are well defined up to conjugacy in $G$. 

\begin{lemma}\label{lem:MN}
Let $K/k$ be an $A_4$-quartic, 
i.e.\ a quartic extension whose normal closure $L/k$ has Galois group $A_4$. Then the Hasse norm principle holds for $K/k$ if and only if there exists a place $v$ of $k$ whose decomposition group $D_v\subset A_4$ contains a copy of $V_4$.     
\end{lemma}

\begin{proof}
In the number field setting, this is a result of Kunyavski\u{\i}, see~\cite{kun:84}, or a special case of~\cite[Proposition~1.8]{mac.new:22}, which is a consequence of~\cite[Theorems~1.1 and 1.3]{mac.new:22}. In fact, these results also hold true in the setting of extensions of global function fields. Many of the results used in their proofs are purely group theoretic or class field theoretic. A few of them assume characteristic zero in order to make use of resolution of singularities, for example the work of Voskresenski\u{\i} in~\cite{vos:70}. However, the required results have in fact been proven in the setting of global fields of arbitrary characteristic, see e.g.~\cite[Corollaire~1]{cths:05}. See~\cite{singh} for a useful collection of extensions to the setting of arbitrary global fields of many results in this area that had previously appeared in the context of number fields.
More precisely, for the extension to this more general setting of~\cite[Theorem~2.1]{mac.new:22}, see~\cite[Section~11.6]{vos:98}; for~\cite[Theorem~2.2]{mac.new:22}, see~\cite[Theorem~5.4]{ct:07}; and for the result on Picard groups used in the proof of~\cite[Theorem~3.3]{mac.new:22}, see~\cite[Proposition~3.1]{ct:07}. 

For the reader's convenience, we also give a direct proof of the result we need here. Tate (\cite[p.\ 198]{cas.fro:67}) has shown that \(\mathfrak{K}(L/k)\) is dual to $\ker(\mathrm{H}^3(G,\ZZ)\to \prod_{v}\mathrm{H}^3(D_v,\ZZ))$
where $G=\Gal(L/k)\cong A_4$ and $v$ runs over all places of $k$. One computes $\mathrm{H}^3(A_4,\ZZ)=\ZZ/2\ZZ$ and uses the fact that $V_4$ is the Sylow $2$-subgroup of $A_4$ along with~\cite[Theorem III.10.3]{bro:82} to deduce that \(\mathfrak{K}(L/k)\hookrightarrow \ZZ/2\ZZ\) is trivial if and only if there exists a place $v$ of $k$ with $V_4\hookrightarrow D_v$. Now observe that $\mathfrak{K}(K/k)$ is killed by $[K:k]=4$ and apply~\cite[Proposition~3.9]{mac.new:22} to obtain an isomorphism $\mathfrak{K}(K/k)\cong\mathfrak{K}(L/k)$.
\end{proof}

\begin{lemma}\label{lem: splitting conditions}
Let $K/k$ be an $A_4$-quartic. Then there is a place $v$ of $k$ with $V_4\hookrightarrow D_v$ if and only if there is a place $v'$ of $K$ above $v$ with $e(v'/v)=f(v'/v)=2$. Moreover, if this is the case
then $D_v\cong V_4$. 
\end{lemma}
\begin{proof}
Let $v$ be a place of $k$ and $w$ a place of $L$ lying above $v$. Since $D_v$ is the Galois group of the extension of local fields $L_w/k_v$, we have a filtration of $D_v$ formed by the ramification groups 
$$D_v\rhd G_0\rhd G_1\rhd G_2\rhd\ldots,$$
where $G_0=I_v$, the quotient $I_v/G_1\hookrightarrow \mathbb{F}_{w}^\times$ is cyclic of order coprime to $p$, and $G_i/G_{i+1}\hookrightarrow (\mathbb{F}_{w},+)$ 
is a $p$-group for $i\geq1$ \cite[Prop. II.10.2]{neu:99}. Therefore, the group $G_1$ is a $p$-group, and since $(p,|G|)=1$ by assumption, it follows that $G_1$ is trivial, hence $I_{v}$ is cyclic. Since $D_{v}/I_{v}$ is the Galois group of the field extension $\mathbb{F}_{w}/\mathbb{F}_{v}$  
by (\ref{ses Dv Iv}), it is also cyclic.
As before, we denote by $H$ the Galois group $\Gal(L/K)$; it is a subgroup of $A_4$ of order $3$.

Now assume that $V_4\hookrightarrow D_v$. Since $D_v$ is a subgroup of $G\cong A_4$, we deduce that $D_v$ is isomorphic to $V_4$ or $A_4$. 
As both $I_v$ and $D_v/I_v$ are cyclic, we find $D_v\cong V_4$, and $I_v\cong C_2$. Therefore, $e(w/v)=f(w/v)=2$ and $D_v\cap H=\{1\}$. 
Now let $v'$ be a place of $K$ below $w$, and note that $K_{v'}\subset L_w$ is the fixed field of the intersection $D_v\cap H\subset A_4$. Since $D_v=V_4$, it follows that $K_{v'}$ is the fixed field of the trivial group, hence $K_{v'}=L_w$. We conclude that $e(v'/v)=f(v'/v)=2$, as required. 

Conversely, let $v'$ be a place of $K$ above $v$ and assume that $e(v'/v)=f(v'/v)=2$. Then $4$ divides $|D_v|$ and by considering the subgroups of $A_4$ we conclude that $V_4\hookrightarrow D_v$.
\end{proof}

\begin{remark}\label{rem:ef}
  It follows from the proof of Lemma~\ref{lem: splitting conditions} that the inertia group at a ramified place is a cyclic subgroup of $A_4$, hence it is either generated by a $3$-cycle or by a product of two disjoint 2-cycles. We refer to the former as type $(3,1)$ and the latter as type $(2,2)$.  In the notation of \cite[\S~2]{bha:07}, the former has splitting symbol \((1^31^{1})\) and the latter has splitting symbol \((1^21^2)\) if \(f(v'/v) = 1\) and \((2^2)\) if \(f(v'/v)=2\). 
\end{remark}

\subsection{The Chebotarev density theorem}\label{sec:chebotarev}

We recall the Chebotarev density theorem for function fields.
Let \(G\) be a finite group.
Let \(X\) and \(Y\) be \(\FF_q\)-schemes and \(\phi \colon X \to Y\) a \(G\)-torsor defined over $\FF_q$. 
Fix an algebraic closure \(\FF_q \subset \overline \FF_q\) and let \(\Frob_q \in \Gal(\overline \FF_q/\FF_q)\) be the Frobenius.
Given \(y \in Y(\FF_q)\), consider the \(\FF_q\)-scheme \(X_y\).
The map \(\phi \colon X_y \to y = \Spec \FF_q\) is a \(G\)-torsor.
Therefore, given \(\overline x \in X_y(\overline \FF_q)\), there is a unique \(g \in G\) such that
\[ \Frob_q (\overline x) = g \cdot \overline x.\]
A different choice of \(\overline x \in X_y(\overline \FF_q)\) changes \(g\) by a conjugate in \(G\).
So the conjugacy class of \(g\) is a well-defined invariant associated to \(\phi\) and \(y\), which we denote by \(\varphi(y)\).
\begin{theorem}[Chebotarev density theorem]
  \label{thm:chebotarev}
  Let \(G\) be a finite group, \(\phi \colon X \to Y\) a \(G\)-torsor defined over \(\FF_q\), and \(\Phi \subset G\) a conjugacy class.
  \begin{enumerate}
      \item  Suppose that \(X\) is geometrically connected.
  Then
  \[ \lim_{m \to \infty} \frac{|\{y \in Y(\FF_{q^m}) \mid \varphi(y) = \Phi\}|}{|Y(\FF_{q^m})|} = \frac{|\Phi|}{|G|}. \]
  \item More generally, suppose that \(Y\) is geometrically connected and \(X\) has \(r\) geometric components, all defined over \(\FF_q\).
  Then
  \[ \lim_{m \to \infty} \frac{|\{y \in Y(\FF_{q^{m}}) \mid \varphi(y) = \Phi\}|}{|Y(\FF_{q^{m}})|} \leq r\frac{|\Phi|}{|G|}, \]
  with equality if and only if $\Phi$ is contained in the stabiliser of a geometric component of $X$.
  \end{enumerate}
\end{theorem}
\begin{proof}
  The first assertion is well known; see, for example, \cite[Th\'eor\`eme 1]{lan:56}.
  It has a short proof using the Lang--Weil bounds; see \cite{mea:18}.
  
  For the second assertion, let \(X_0 \subset X\) be a geometrically connected component.
  Let \(H \subset G\) be the stabiliser of \(X_0\).
  Then \(X_0\) is geometrically connected, \(X_0 \to Y\) is an \(H\)-torsor, and the index of \(H\) in \(G\) is \(r\).

  Associated to a \(y \in Y(\FF_{q^m})\), we now have two conjugacy classes.
  The first is a conjugacy class of \(G\) from the \(G\)-torsor \(X \to Y\).
  The second is a conjugacy class of \(H\) from the \(H\)-torsor \(X_0 \to Y\).
  Denote the first by \(\varphi^G(y)\) and the second by \(\varphi^H(y)\).
  We claim that \(\varphi^H(y) \subset \varphi^G(y)\).
  To see this, recall that the definition of \(\varphi^G(y)\) involves a choice of an arbitrary \(\overline x \in X(\overline \FF_q)\) over \(y \in Y(\FF_{q^m})\).
  If we take \(\overline x\) to be in \(X_0\), then the unique \(g \in G\) such that \(\Frob_{q^m}(\overline x) = g \cdot \overline x\) lies in \(H \subset G\).
  By definition, \(\varphi^G(y)\) is the conjugacy class of \(G\) represented by \(g\) and \(\varphi^H(y)\) is the conjugacy class of \(H\) represented by \(g\).
  So \(\varphi^H(y) \subset \varphi^G(y)\).

  The subset \(\Phi \cap H \subset H\) is preserved by conjugation by elements of \(H\).
  Write it as a (possibly empty) union of distinct conjugacy classes of \(H\), say
  \[\Phi \cap H = \Phi_1 \sqcup \dots \sqcup \Phi_n.\]
  Since \(\varphi^H(y) \subset \varphi^G(y)\), we see that \(\varphi^G(y) = \Phi\) if and only if \(\varphi^H(y) = \Phi_i\) for some \(i\).
  By the Chebotarev density theorem applied to the \(H\)-torsor \(X_0 \to Y\), we have
  \[ \lim_{m \to \infty} \frac{|\{y \in Y(\FF_{q^{m}}) \mid \varphi^H(y) = \Phi_i\}|}{|Y(\FF_{q^{m}})|} = \frac{|\Phi_i|}{|H|}.\]
  Since \(\varphi^G(y) = \Phi\) is equivalent to \(\varphi^H(y) = \Phi_i\) for some \(i\), we get
  \[ \lim_{m \to \infty} \frac{|\{y \in Y(\FF_{q^{m}}) \mid \varphi^G(y) = \Phi\}|}{|Y(\FF_{q^{m}})|} = \frac{|\Phi_1|+\dots+|\Phi_{n}|}{|H|} = \frac{|\Phi\cap H|}{|H|}.\]
  Recalling that the index of \(H\) in \(G\) is \(r\), we have
  \[\frac{|\Phi \cap H|}{|H|} \leq \frac{|\Phi|}{|H|} = r \frac{|\Phi|}{|G|},\]
  with equality if and only if $\Phi\subset H$.
  The proof is now complete.
\end{proof}

Let \(X\) and \(Y\) be \(\FF_q\)-schemes, \(\phi \colon X \to Y\) a \(G\)-torsor defined over $\FF_q$, and \(\Lambda \subset G\) a subgroup. Let \(Z\) be the quotient of \(X\) by \(\Lambda\).
This quotient exists as a scheme by \cite[Expose~V, Proposition~3.1]{03}.
Since \(X \to Y\) is finite and \'etale, so is the induced map \(\phi \colon Z \to Y\).   
\begin{proposition}\label{prop:chebapply}
  In the set-up above, let \(\Phi \subset G\) be the union \(\Phi = \bigcup_{h \in G} h^{-1}\Lambda h\).
  Then
  \[ \phi(Z(\FF_q)) = \{y \in Y(\FF_q) \mid \varphi_y \subset \Phi\}.\]
  Furthermore, if \(Y\) is geometrically connected and \(X\) has \(r\) geometric components, all defined over \(\FF_q\), then
  \[ \lim_{m \to \infty} \frac{|\phi(Z(\FF_{q^{m}}))|}{|Y(\FF_{q^{m}})|} \leq r \frac{|\Phi|}{|G|},\]
     with equality if and only if $\Phi$ is contained in the stabiliser of a geometric component of $X$.
\end{proposition}
\begin{proof}
  Let \(y \in Y(\FF_q)\).
  Let \(X_y(\overline \FF_q)\) be the set of geometric points of \(X\) over \(y\).
  Choose \(\overline x \in X_y(\overline \FF_q)\).
  Since \(X \to Y\) is a \(G\)-torsor, we have
  \[ X_y(\overline \FF_q) = \{g\cdot  \overline x \mid g \in G\}.\]
  Since \(Z = X/\Lambda\), the elements of \(Z_y(\overline \FF_q)\)
  correspond to right \(\Lambda\)-cosets
  \[ Z_y(\overline \FF_q) = \{\Lambda h\cdot \overline x \mid h \in G\}.\]

  Let \(g \in G\) be such that \( \Frob_q (\overline x) = g \cdot \overline x\).
  By definition, \(\varphi(y)\) is the conjugacy class of \(g\).
  The action of \(\Frob_q\) on \(Z_y(\overline \FF_q)\) sends the point corresponding to the coset \(\Lambda h \cdot \overline x\) to the point corresponding to \(\Lambda h \cdot g \overline x\).
  This point is fixed by \(\Frob_q\) if and only if \(\Lambda h = \Lambda hg\), which in turn is equivalent to \(g \in h^{-1}\Lambda h\).
  We conclude that \(y \in Y(\FF_q)\) lies in the image of \(Z(\FF_q)\) if and only if there exists an \(h \in G\)  and \(g \in h^{-1}\Lambda h\) such that \(\varphi_y\) is the conjugacy class of \(g\).
  This statement is clearly equivalent to \(\varphi_y \subset \Phi\).
  We have thus proved the first assertion.
  Given the first assertion, the second follows from \Cref{thm:chebotarev}. 
\end{proof}

\section{Hurwitz spaces and the Hasse norm principle}\label{sec:hurwitz}
In this section, we define a Hurwitz space of covers of the projective line, and a particular covering of this space that accounts for the failure of the Hasse norm principle.

Fix 
\(n_1\in\mathbb{Z}_{>0}\) and 
\(n_2\in\mathbb{Z}_{\geq 0}\).
For
\(d\in\mathbb{Z}_{\geq 0}\), we interpret \(\PP^d\) as the moduli space of effective divisors of degree \(d\) on \(\PP^1\).
Let \(U(n_1,n_2) \subset \PP^{n_1} \times \PP^{n_2}\) be the open subscheme whose \(S\)-points for a scheme $S$ are \((B_1,B_2)\), where
\begin{enumerate}
\item \(B_i \subset \PP^{1}_S\) is finite and \'etale over \(S\) of degree \(n_i\), and 
 \item \(B_1\) and \(B_2\) are disjoint. 
 \end{enumerate}
 Let \(\cU(n_1,n_2)\) be the functor represented by \(U(n_1,n_2)\).
 
An \emph{\(A_4\)-quartic cover of $\PP^1$} over a scheme \(S\) consists of \[(C \xrightarrow{\pi} S, B_1, B_2, C \xrightarrow{f} \PP^1_S),\] 
where
\begin{enumerate}
\item \(\pi\) is a proper smooth morphism whose geometric fibres are connected and of dimension 1,
\item \(B_1, B_2 \subset \PP^1_S\) are disjoint subschemes, finite and \'etale over \(S\) of degrees \(n_1\) and \(n_2\), respectively,
\item \(f\) is a finite flat morphism of degree \(4\) that is \'etale over the complement of \(B_1 \cup B_2\);
\item for every geometric point \(s \to S\), the degree 4 map \(f_s \colon C_s \to \PP^1_s\) has ramification type \((3,1)\) over every point of \(B_{1,s}\) and type \((2,2)\) over every point of \(B_{2,s}\).
  In the notation of \cite[\S~2]{bha:07}, the splitting symbol is \((1^31^1)\) over every (geometric) point of \(B_{1,s}\) and \((1^21^2)\) over every (geometric) point of \(B_{2,s}\).
\end{enumerate}
We remind the reader that our schemes are over \(\ZZ[1/6]\), and therefore a cover of degree \(4\) is automatically tame.
\begin{remark}\label{rem:A4}
  Connected degree 4 covers with ramification of type \((3,1)\) over at least one point and of type \((2,2)\) over zero or more points are precisely the covers with Galois group \(A_4\).
  Indeed, let \(k\) be an algebraically closed field of characteristic greater than 3, and let \[(C\to \Spec{k}, B_1,B_2,  C \to \PP^1_k)\] 
  be an \(A_4\)-quartic cover of $\PP^1_k$.
  Then the Galois group of \(k(C) / k(\PP^1)\) is a transitive subgroup of the symmetric group \(S_4\) generated by at least one 3-cycle (since \(n_1 = \deg B_1 > 0\)) and zero or more \((2,2)\)-cycles.
  The only such subgroup is the alternating group \(A_4\).
  Conversely, any degree 4 cover \(f \colon C \to \PP^1_k\) with Galois group \(A_4\) must have ramification types \((3,1)\) or \((2,2)\), and \((3,1)\) must appear at least once.
  \end{remark}

Let \((C \to S, B_1, B_2, C \xrightarrow{f} \PP^1_S)\) and \((C' \to S, B_1', B_2', C' \xrightarrow{f'} \PP^1_S)\) be two \(A_4\)-quartic covers of \(\PP^1_S\).
We say that they are \emph{isomorphic} if for \(i = 1,2\) we have \(B_i = B_i'\) and there exists an \(S\)-isomorphism \(\psi \colon C \to C'\) such that the following diagram commutes:
\[
  \begin{tikzcd}
    C \ar{r}{\psi}\ar{d}{f}& C' \ar{d}{f'}\\
    \PP^1_S \ar[equal]{r}& \PP^1_{S}.
  \end{tikzcd}
\]

Let \(\cH(n_1,n_2)\) be the contravariant functor from \(\Schemes\) to \(\Sets\) that sends \(S\) to the set of isomorphism classes of \(A_4\)-quartic covers of \(\PP^1_S\). 
\begin{remark}
  A more ``correct'' formulation of the functor of \(A_4\)-quartic covers would have been as a functor valued in groupoids.
  However, it is easy to check that \(A_4\)-quartic covers of \(\PP^1\) have trivial automorphism groups.
  So the  groupoid-valued functor is equivalent to a set-valued functor.
\end{remark}

\begin{theorem}\label{thm:HoverU}
  The functor \(\cH(n_1,n_2)\) is represented by a smooth quasi-projective \(\ZZ[1/6]\)-scheme \(H(n_1,n_2)\) of relative dimension $n_1+n_2$. The forgetful natural transformation \(\cH(n_1,n_2) \to \cU(n_1,n_2)\) is represented by a finite \'etale morphism
  \[ H(n_1,n_2) \to U(n_1,n_2).\]
\end{theorem}
\begin{proof}
  Follows from arguments analogous to the construction of the Hurwitz scheme of simple coverings due to Fulton \cite[\S~6]{ful:69} (see also, \cite[\S~4]{wew:98}).
  We note that \(U(n_1,n_2) \to \Spec \ZZ[1/6]\) is smooth of relative dimension \(n_1+n_2\).
  Therefore, \(H(n_1,n_2) \to \Spec \ZZ[1/6]\) is also smooth of the same relative dimension.
\end{proof}

Let \[(\mathcal{C} \xrightarrow{\pi} H(n_1,n_2), \mathcal{B}_1, \mathcal{B}_2, \mathcal{C} \xrightarrow{f} \PP^1 \times H(n_1,n_2))\] be the universal \(A_4\)-quartic cover of \(\PP^1\) over \({H(n_1,n_2)}\).
Let \( {\mathcal C}_2 \subset \mathcal{C}\) be the reduced subscheme underlying \(f^{-1}(\mathcal{B}_2)\).
\begin{proposition}
  As divisors on \(\mathcal{C}\), we have
  \[ 2  \cdot {\mathcal C}_2 = f^{*} \mathcal B_2.\]
  Furthermore, the map \(f \colon {\mathcal C}_2 \to \mathcal{B}_2\) is finite \'etale of degree \(2\).
\end{proposition}
\begin{proof}
  This is a simple local calculation using the structure theorem for finite tame covers \cite[\S~4]{ful:69}.
\end{proof}

\begin{remark}\label{rem:reduced}
  Let \(S\) be a reduced scheme and let \((C \xrightarrow{\pi} S, B_1,B_2,  C \xrightarrow{f} \PP^1_S)\) be an \(A_4\)-quartic cover of \(\PP^1\) over \(S\).
  Let \(S \to H(n_1,n_2)\) be the corresponding morphism.
  Then \(\mathcal{C}_2 \times_{H(n_1,n_2)} S\) is simply the reduced scheme underlying \(f^{-1}(B_2)\).
  Indeed, both are reduced schemes with the same support.
\end{remark}

Let \(S\) be a scheme and let \(X_1, X_2,\) and \(Y\) be \(S\)-schemes with morphisms \(X_1 \to Y\) and \(X_2 \to Y\).
By \(\operatorname{Isom}_S(X_1 \to Y, X_2 \to Y)\) we denote the functor that sends an \(S\)-scheme \(T\) to the set of isomorphisms
\[ i \colon X_{1,T} \to X_{2,T}\]
that make the following diagram commute
\[
  \begin{tikzcd}
    X_{1,T}\ar{d}\ar{r}{i} & X_{2,T} \ar{d}\\
    Y_T\ar[equal]{r} & Y_T.
  \end{tikzcd}
\]
If $Y=S$, we write \(\operatorname{Isom}_S(X_1, X_2 )\) for \(\operatorname{Isom}_S(X_1\to S, X_2\to S )\). 

We let \(\cH^{\rm fail}(n_{1},n_2)\) be the functor 
\[ \cH^{\rm fail}(n_1,n_2) = \operatorname{Isom}_{H(n_1,n_2)}(\{\pm 1\} \times \mathcal{B}_2 \to \mathcal{B}_2, \mathcal{C}_2 \to \mathcal{B}_2).\]
Let us pause to understand the functor over \(S = \Spec k\) for an algebraically closed field \(k\).
In this case, an element of \(\cH^{\rm fail}(n_1,n_2)(S)\) is an \(A_4\)-quartic cover of \(\PP^1_k\), say  \((C\to \Spec{k}, B_1,B_2, C \xrightarrow{f} \PP^1_k)\) together with an isomorphism
\[\{\pm 1\} \times B_2 \to \mathcal{C}_{2,k}\]
over \(B_2\).
From \autoref{rem:reduced}, we know that \(C_2 = \mathcal{C}_{2,k}\) is the reduced pre-image of \(B_2\) in \(C\).
An isomorphism over \(B_2\) from \(\{\pm 1\} \times B_2\) to \(C_2\) is simply a labelling by \(\{\pm 1\}\) of the two points of \(C_2\) over every point of \(B_2\).
Thus, \(\cH^{\rm fail}(n_1,n_2)(S)\) is the set of isomorphism classes of \(A_4\)-quartic covers \((C\to \Spec{k}, B_1,B_2, C \xrightarrow{f} \PP^1_k)\) of \(\PP^1_k\) together with a labelling of the two points of \(C\) over each point of \(B_2\).

\begin{proposition}
  The functor \(\cH^{\rm fail}(n_1,n_2)\) is represented by an \(H(n_1,n_2)\)-scheme \(H^{\rm fail}(n_1,n_2)\).
  The natural map \(H^{\rm fail}(n_1,n_2) \to H(n_1,n_2)\) is finite and \'etale of degree \(2^{n_2}\).
\end{proposition}
\begin{proof}
  The representability of \(\operatorname{Isom}\) follows from standard arguments using the Hilbert scheme (for example, \cite[\S~4 Variantes]{gro:61}).
  In fact, it is even simpler in our case since the morphisms involved are finite.

  To see that the natural map is finite \'etale of the given degree, we make an \'etale base change \(S \to H(n_1,n_2)\) such that we have an isomorphism \(\{1,\dots, n_2\} \times S \to \mathcal{B}_{2,S}\) over \(S\) and an isomorphism \(\{\pm 1\} \times \mathcal{B}_{2,S} \to \mathcal{C}_{2,S}\) over \(\mathcal{B}_{2,S}\).
  Let \(F\) be the subset of the permutations of \(\{\pm 1\} \times \{1, \dots, n_2\}\) that commute with the projection to \(\{1,\dots,n_2\}\).
  Then, it is easy to see that \(F = (\ZZ/2\ZZ)^{n_2}\) and \(H^{\rm fail}(n_1,n_2)_S = F \times S\).
\end{proof}
We now relate the map \(H^{\rm fail}(n_1,n_2) \to H(n_1,n_2)\) to the failure of the Hasse norm principle.
\begin{theorem}\label{thm:failset}
  Let \(k\) be a finite field of characteristic \(\neq 2,3\).
  Let \((C, B_1, B_2,  C \xrightarrow{f} \PP^1_k)\) be an \(A_4\)-quartic cover of \(\PP^1\) over \(k\).
  The Hasse norm principle fails for \(f\) if and only if the \(k\)-point of \(H(n_1,n_2)\) represented by \(f\) lies in the image of a \(k\)-point of \(H^{\rm fail}(n_1,n_2)\).
\end{theorem}

\begin{proof}
  Let \(K = k(C)\) be the function field of \(C\) and \(k(\PP^1) = k(t)\) the function field of \(\PP^1_k\).
  Let \(C_2 \subset C\) be the reduced pre-image of \(B_2\).
    
  \Cref{lem:MN,lem: splitting conditions} show that the Hasse norm principle fails for $K/k(t)$ if and only if there is no place $v'$ of $K$ lying above a place $v$ of $k(t)$ with $e(v'/v)=f(v'/v)=2$.
  A place \(v\) of \(k(t)\) corresponds to a closed point \(p\) of \(\PP^1_k\); a place of \(K\) over \(v\) corresponds to a closed point \(p'\) of \(C\) over \(p\).
  If \(p\) lies in the complement of \(B_1 \cup B_2\), then \(e(p'/p) = 1\) since \(C \to \PP^1_k\) is unramified there.
  If \(p\) lies in \(B_1\), then \(e(p'/p) = 1\) or \(3\) since the ramification type over points of \(B_2\) is \((3,1)\).

  For a closed point \(p \in B_2\), we have $\sum e(p'/p)f(p'/p)=4$, where the sum runs over all closed points \(p'\) of \(C_2\) lying over $p$.
  Since the ramification type over points of \(B_2\) is $(2,2)$, we have $e(p'/p)=2$ for all \(p'\) over \(p\), and hence $\sum f(p'/p)=2$. This means we are in one of the following two cases: 
\begin{enumerate}
    \item \label{succeed} there is a unique $p'$ lying over \(p\) with $f(p'/p)=2$; or
    \item \label{fail} there are two closed points $p'_1$ and $p'_2$ of $C$ lying over \(p\) with $f(p'_1/p)=f(p'_2/p)=1$. 
\end{enumerate}

If Case~\ref{succeed} occurs for some closed point \(p\) of \(B_2\), then the Hasse norm principle holds for \(f\) by \Cref{lem: splitting conditions}. 
In this case, the residue field of \(C\) at \(p'\) is a degree 2 extension of the residue field of \(\PP^1\) at \(p\).
In particular, there is no isomorphism \(\{\pm 1\} \times p \to p'\) and, therefore, no isomorphism \(\{\pm 1\} \times B_2 \to C_2\).
Therefore, the \(k\)-point of \(H(n_1,n_2)\) represented by \(f\) is not the image of a \(k\)-point of \(H^{\rm fail}(n_{1},n_2)\).

If Case~\ref{fail} occurs for every closed point \(p\) of \(B_2\), then the Hasse norm principle fails for \(f\) by \Cref{lem: splitting conditions}.
In this case, the labelling \(+1 \mapsto p_1', -1 \mapsto p_2'\) of the two points of \(C\) over \(p\) gives an isomorphism \(\{\pm 1\} \times p \to C_{2,p}\). 
Letting \(p\) vary over the closed points of \(B_2\) gives an isomorphism \(\{\pm 1\} \times B_2 \to C_2\).
This isomorphism gives a \(k\)-point of \(H^{\rm fail}(n_1,n_2)\) that maps to the \(k\)-point represented by \(f\).
\end{proof}

\section{Hurwitz spaces for the Chebotarev density theorem}\label{sec:morehurwitz} 
Let \(\phi \colon H^{\rm fail}(n_1,n_2) \to H(n_1,n_2)\) be the natural forgetful map of Hurwitz spaces defined in \Cref{sec:hurwitz}.
Our goal is to use the Chebotarev density theorem to estimate the ratio
\[ \frac{|\phi(H^{\rm fail}(n_1,n_2)(\FF_q))|}{|H(n_1,n_2)(\FF_q)|}. \]
By \Cref{thm:failset}, this ratio represents the proportion 
of \(A_4\)-quartic covers of \(\PP^1_{\FF_q}\) for which the Hasse norm principle fails.
To estimate this ratio, we express \(H^{\rm fail}(n_1,n_2)\) as the quotient of a \(G\)-torsor over \(H(n_1,n_2)\).
This \(G\)-torsor is another Hurwitz space, with some more decoration.
We now define it.

Recall that
\[(\mathcal{C} \xrightarrow{\pi} H(n_1,n_2), \mathcal{B}_1, \mathcal{B}_2, \mathcal{C} \xrightarrow{f} \PP^1 \times H(n_1,n_2))\]
is the universal object over \(H(n_1,n_2)\).
The map \(\mathcal{B}_2 \to H(n_1,n_2)\) is \'etale of degree \(n_2\).
Set
\begin{equation}\label{eqn:Htilde}
  \widetilde H(n_1,n_2) = \operatorname{Isom}_{H(n_1,n_2)}(\{1,\dots, n_2\} \times H(n_1,n_2), \mathcal{B}_2),
\end{equation}
and
\begin{equation}\label{eqn:failproduct}
    \widetilde H^{\rm fail} (n_1,n_2) = H^{\rm fail}(n_1,n_2) \times_{H(n_1,n_2)} \widetilde H(n_1,n_2).
  \end{equation}
The following pull-back diagram summarises the relationships between the spaces defined so far.
All the arrows are finite and \'etale of the indicated degree.
  \begin{equation}\label{eqn:fourspaces-degrees}
  \begin{tikzcd}
    \widetilde H^{\rm fail}(n_1,n_{2})\ar{d}[left]{2^{n_2}}\ar{r}{n_2!} & H^{\rm fail}(n_1,n_2)\ar{d}{2^{n_2}} \\
    \widetilde H(n_1,n_2)\ar{r}[below]{n_{2}!} & H(n_1,n_2).
  \end{tikzcd}
\end{equation}
For a scheme \(S\), let us also recall the \(S\)-points of each of the spaces above.
\begin{description}
\item[\(H(n_1,n_2)\)]\label{eqn:descH} An isomorphism class of an \(A_4\)-quartic cover of \(\PP^1_S\), say \[(C \to S, B_1,B_2,  C \xrightarrow{f} \PP^1_{S}).\]
  Let \(C_2 \subset C\) be the pull-back of \(\mathcal{C}_2\).
\item[\(H^{\rm fail}(n_1,n_2)\)]
  The data above together with an \(S\)-isomorphism \(\{\pm 1\} \times B_2 \to C_2\) compatible with the projection to \(B_2\).
\item[\(\widetilde H(n_1,n_2)\)] A point of \(H(n_1,n_2)\) as above together with an \(S\)-isomorphism \[\{1,\dots,n_2\} \times S \to B_2.\]
\item[\(\widetilde H^{\rm fail}(n_1,n_2)\)] A point of \(H(n_1,n_2)\) as above together with \(S\)-isomorphisms
  \[ \nu \colon \{\pm 1\} \times \{1,\dots,n_2\} \times S \to C_2 \text{ and } \overline\nu \colon \{1,\dots,n_2\} \times S \to B_2\]
  such that the following diagram commutes
  \[
    \begin{tikzcd}
      \{\pm 1\} \times \{1, \dots, n_2\} \times S \ar{r}{\nu}\ar{d} & C_2 \ar{d}{f}\\
      \{1,\dots,n_2\} \times S \ar{r}{\overline \nu} & B_2.
    \end{tikzcd}
    \]
\end{description}

Having described the \(S\)-points, we now define natural group actions on the spaces above under which three of the arrows in \eqref{eqn:fourspaces-degrees} become torsors.
First, the natural action of the symmetric group \(S_{n_2}\) on \(\{1,\dots,n_2\}\) induces an action on \(\widetilde H(n_1,n_2)\).
Under this action, \(\widetilde H(n_1,n_2) \to H(n_1,n_2)\) is an \(S_{n_2}\)-torsor.
By pull-back, \(\widetilde H^{\rm fail}(n_1,n_2) \to H^{\rm fail}(n_1,n_2)\) is also an \(S_{n_2}\)-torsor.

Let \(G \subset \operatorname{Perm}(\{\pm 1\} \times \{1,\dots,n_2\})\) be the subgroup consisting of permutations \(\alpha\) such that for all \(i\in\{1,\dots, n_2\}\), the second coordinates of \(\alpha(+1,i)\) and \(\alpha(-1,i)\) agree.
We have a natural map \(G \to S_{n_2}\) that sends \(\alpha\) to the permutation \(\overline \alpha\) that takes \(i\) to the second coordinate of \(\alpha(\pm 1,i)\).
This map is clearly surjective.
In fact, it has a left inverse that sends \(\sigma \in S_{n_2}\) to the \(\alpha\) that sends \((\pm 1,i)\) to \((\pm 1,\sigma(i))\).
It is easy to see that the kernel of \(G \to S_{n_2}\) is \((\ZZ/2\ZZ)^{n_2}\), and \(G\) is the semidirect product
\[ G = (\ZZ/2\ZZ)^{n_2} \rtimes S_{n_2}\]
for the permutation action of \(S_{n_2}\) on \((\ZZ/2\ZZ)^{n_2}\).  

We have an action of \(G\) on \(\widetilde H^{\rm fail}(n_1,n_2)\).
The element \(\alpha \in G\) acts by \(\nu \mapsto \nu \circ \alpha^{-1}\) and \(\overline \nu \mapsto \overline \nu \circ \overline \alpha^{-1}\).
The map \(\widetilde H^{\rm fail}(n_1,n_2) \to H(n_1,n_2)\) is a \(G\)-torsor under this action.
Observe that the action of the subgroup \(S_{n_2} \subset G\) on \(\widetilde H^{\rm fail}(n_1,n_2)\) is precisely the action induced from the action on \(\widetilde H(n_1,n_2)\).

We now redraw the diagram \eqref{eqn:fourspaces-degrees} and decorate  each arrow with
the group under which it is a torsor.
(The right vertical arrow is not a torsor.) 
\begin{equation}\label{eqn:fourspaces}
  \begin{tikzcd}
    \widetilde H^{\rm fail}(n_1,n_{2})\ar{d}[left]{(\ZZ/2\ZZ)^{n_2}}\ar{r}{S_{n_2}}\ar{rd}{G} & H^{\rm fail}(n_1,n_2)\ar{d} \\
    \widetilde H(n_1,n_2)\ar{r}[below]{S_{n_2}} & H(n_1,n_2).
  \end{tikzcd}
\end{equation}

We have expressed \(H^{\rm fail}(n_1,n_2)\) as the quotient of a \(G = (\ZZ/2\ZZ)^{n_2} \rtimes S_{n_2}\)-torsor by the action of \(S_{n_2}\).
To apply the Chebotarev density theorem, we must now understand two quantities:
\begin{enumerate}
\item the number of elements of \(G\) that are conjugate to an element of \(S_{n_2} \subset G\), and
\item the number of geometric components of \(\widetilde{H}^{\rm fail}(n_{1},n_2)\) over each geometric component of \(H(n_1,n_2)\). 
  \end{enumerate}

  We take up the first issue here and the second in the next section.
  To lighten notation, set \(\ell = n_2\) so that \(G = (\ZZ/2\ZZ)^{\ell} \rtimes S_{\ell}\).
  Let \(\Phi \subset G\) be the union \(\Phi = \bigcup_{h \in G} h S_\ell h^{-1}\).
  Let \(s(\ell,i)\) be the unsigned Sterling number of the first kind.
  This is the number of elements in \(S_{\ell}\) whose disjoint cycle decomposition has exactly \(i\) cycles.
  \begin{proposition}\label{prop:sterling}
  An element \((u,\sigma) \in G\) lies in \(\Phi\) if and only if there exists \(w \in (\ZZ/2\ZZ)^\ell\) such that \(w + \sigma(w) = u\).
  Consequently, the size of \(\Phi\) is
  \[ |\Phi| = \sum_{i = 1}^\ell 2^{\ell-i} s(\ell,i).\]
\end{proposition}

\begin{proof}
  Consider an arbitrary \((w,\tau) \in G\).
  We have
  \[ (w, \tau) (u, \sigma) (w, \tau)^{-1} = (w + \tau(u)+ \tau\sigma\tau^{-1}(w), \tau\sigma\tau^{-1}). \]
  This element lies in \(S_\ell\) if and only if
  \[ w + \tau(u) + \tau\sigma\tau^{-1}(w) = 0.\]
  Apply \(\tau^{-1}\) and replace \(\tau^{-1}(w)\) by \(w\) to get
  \[ w + \sigma(w) = u.\]

  We now find the size of \(\Phi\).
  Let us first fix \(\sigma \in S_\ell\) and find the number of \(u\) such that \((u,\sigma) \in \Phi\).
  We have proved that \((u, \sigma) \in \Phi\) if and only if \(u \in (\ZZ/2\ZZ)^\ell\) is in the image of the linear map \(w \mapsto w + \sigma(w)\).
  Let \(e_1, \dots, e_\ell \in (\ZZ/2\ZZ)^\ell\) be the standard basis vectors.
  Given \(T \subset \{1, \dots, \ell\}\), set \(e_{T} = \sum_{t \in T}e_t\).
  It is easy to check that the kernel of the map \(w \mapsto w + \sigma(w)\) is spanned by the vectors \(e_{T}\) as \(T\) ranges over the cycles in the disjoint cycle decomposition of \(\sigma\).
  Thus, if \(\sigma\) has \(i\) cycles, then the kernel of \(w \mapsto w + \sigma(w)\) has dimension \(i\), and hence its image has dimension \(\ell-i\).
  We conclude that the number of \(u\) such that \((u,\sigma) \in \Phi\) is \(2^{\ell-i}\).
  We now take the sum as \(\sigma\) varies.
\end{proof}

Let us estimate the quantity in \Cref{prop:sterling}.
For functions \(f, g \colon \mathbb{Z}_{\geq 0} \to \mathbb{R}\), we write \(f \sim g\) if \(\lim_{n \to \infty}(f(n)/g(n))\) is a positive (finite) number.
\begin{proposition}\label{prop:sterlingapprox}
  Let \(b_\ell = \sum_{i = 1}^\ell 2^{-i} s(\ell,i)/\ell!\).
  Then \(b_\ell < \ell^{-1/2}\) and \(b_\ell \sim \ell^{-1/2}\).
\end{proposition}
\begin{proof}
  The unsigned Sterling numbers satisfy the identity
  \[ (1-t)^{-x} = \sum_{\ell=0}^{\infty} \sum_{i = 1}^{\ell} \frac{s(\ell,i)}{\ell!} t^\ell x^i.\]
  Substituting \(x = 1/2\) gives
  \[ (1-t)^{-1/2} = \sum_{\ell=0}^{\infty} \sum_{i = 1}^{\ell} \frac{s(\ell,i)}{\ell!} t^\ell 2^{-i}.\]
  Taking the \(\ell\)-th derivative of both sides and setting \(t = 0\) gives
  \[ \frac12 \cdot \frac32 \cdots \frac{2\ell-1}{2} = \sum_{i=1}^\ell s(\ell,i)2^{-i}.\]
  Dividing both sides by \(\ell!\) yields
  \begin{equation*}
    \left(1 - \frac{1}{2} \right) \left(1 - \frac{1}{4}\right) \cdots \left(1 - \frac{1}{2\ell}\right) = b_\ell.
  \end{equation*}
  Since \(1-x < \exp(-x)\) for \(x > 0\), we get
  \begin{equation}\label{eqn:firstestimate}
    b_\ell < \exp\left(-\sum_{i=1}^\ell \frac{1}{2i}\right).
  \end{equation}
  Furthermore, since
\begin{equation}\label{eqn:secondestimate}
    \log \ell  < \sum_{i=1}^\ell 1/i,
  \end{equation}
  the right-hand side of (\ref{eqn:firstestimate}) is bounded above by  \[\exp\left(-\frac{1}{2}\log \ell\right) = \ell^{-1/2}.\]
  It is standard that the ratio of the two sides of \eqref{eqn:firstestimate} and the difference of the two sides of \eqref{eqn:secondestimate} both approach  
  constants as \(\ell \to \infty\).
  It follows that \(b_\ell \sim \ell^{-1/2}\).
\end{proof}

\section{Connected components of the Hurwitz spaces}\label{sec:comphurwitz}
We now turn to the number of geometric components of the Hurwitz spaces and their fields of definition.
We have the following comparison theorem (see also \cite[Lemma~2.3.5]{lan.lev:25-1} and \cite[Lemma~10.3]{liu.woo.zur:24}).
\begin{proposition}\label{prop:components-complex}
  Let \(H\) be any finite \'etale cover of \(U(n_1,n_2)\), for example, any of the four spaces in \eqref{eqn:fourspaces}. 
  Reduction to characteristic \(p\) gives a bijection between the geometric components of \(H \otimes \QQ\) and the geometric components of \(H \otimes \FF_p\).
\end{proposition}
\begin{proof}
  We have a morphism \(\PP^{n_1} \times \PP^{n_2} \to \PP^{n_1+n_2}\) that sends a pair of effective divisors \((B_1,B_2)\) of degrees \(n_1\) and \(n_2\) to the divisor \(B_1 + B_2\).
  Let \(\Delta \subset \PP^{n_1+n_2}\) be the discriminant divisor.
  Then \(U(n_1,n_2) \to \PP^{n_1+n_2} - \Delta\) is a finite \'etale cover.
  Thus, \(H\) is a finite \'etale cover of \(\PP^{n_1+n_2} - \Delta\).
  The result now follows from Fulton's reduction theorem \cite[Theorem~3.3, Lemma~A.3]{ful:69}, following an argument similar to \cite[Corollary 7.5]{ful:69}.
\end{proof}

\begin{figure}
  \centering
  \begin{tikzpicture}[scale=0.5]
    \draw[fill] (-2,0) circle (0.1)  (-1,0) circle (0.1) (0,0) node {\(\cdots\)} node[above=0.2cm] {\(B_1\)} (1,0) circle (0.1) (2,0) circle (0.1);
    \coordinate (base) at (3.5,-3);

    \draw (base) -- (-2,-0.4);
    \draw [postaction={decorate}, decoration={markings, mark=at position 0.5 with \arrow{To[reversed]}}]
    plot[smooth, tension=1] coordinates {(-2,-0.4) (-1.6,0) (-2,0.4) (-2.4,0) (-2,-0.4)};

    \draw (base) -- (-1,-0.4);
    \draw [postaction={decorate}, decoration={markings, mark=at position 0.5 with \arrow{To[reversed]}}]
    plot[smooth, tension=1] coordinates {(-1,-0.4) (-0.6,0) (-1,0.4) (-1.4,0) (-1,-0.4)};

    \draw (base) -- (1,-0.4);
    \draw [postaction={decorate}, decoration={markings, mark=at position 0.5 with \arrow{To[reversed]}}]
    plot[smooth, tension=1] coordinates {(1,-0.4) (1.4,0) (1,0.4) (0.6,0) (1,-0.4)};

    \draw (base) -- (2,-0.4);
    \draw [postaction={decorate}, decoration={markings, mark=at position 0.5 with \arrow{To[reversed]}}]
    plot[smooth, tension=1] coordinates {(2,-0.4) (2.4,0) (2,0.4) (1.6,0) (2,-0.4)};

    \begin{scope}[xshift=7cm]
    \draw[fill] (-2,0) circle (0.1)  (-1,0) circle (0.1) (0,0) node {\(\cdots\)} node[above=0.2cm] {\(B_2\)} (1,0) circle (0.1) (2,0) circle (0.1);
    \draw (base) -- (-2,-0.4);
    \draw [postaction={decorate}, decoration={markings, mark=at position 0.5 with \arrow{To[reversed]}}]
    plot[smooth, tension=1] coordinates {(-2,-0.4) (-1.6,0) (-2,0.4) (-2.4,0) (-2,-0.4)};

    \draw (base) -- (-1,-0.4);
    \draw [postaction={decorate}, decoration={markings, mark=at position 0.5 with \arrow{To[reversed]}}]
    plot[smooth, tension=1] coordinates {(-1,-0.4) (-0.6,0) (-1,0.4) (-1.4,0) (-1,-0.4)};

    \draw (base) -- (1,-0.4);
    \draw [postaction={decorate}, decoration={markings, mark=at position 0.5 with \arrow{To[reversed]}}]
    plot[smooth, tension=1] coordinates {(1,-0.4) (1.4,0) (1,0.4) (0.6,0) (1,-0.4)};

    \draw (base) -- (2,-0.4);
    \draw [postaction={decorate}, decoration={markings, mark=at position 0.5 with \arrow{To[reversed]}}]
    plot[smooth, tension=1] coordinates {(2,-0.4) (2.4,0) (2,0.4) (1.6,0) (2,-0.4)};
  \end{scope}
  
  \end{tikzpicture}
  \caption{We describe a point of \(H(n_1,n_2)_{\CC}\) over a point of \(U(n_1,n_2)_{\CC}\) in terms of monodromy data consisting of elements of \(S_4\) associated to each loop as shown.  We describe a point of \(H^{\rm fail}(n_1,n_2)_{\CC}\) similarly, using decorated monodromy data.}\label{fig:diagram}
\end{figure}
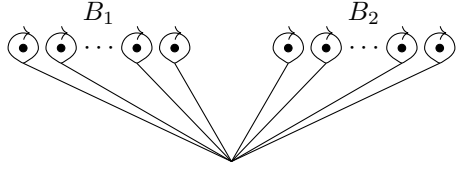

We analyse the connected components of the four spaces in \eqref{eqn:fourspaces} over \(\CC\).
We take a two-step approach.
In the first step, we understand the components of \(\widetilde H(n_1,n_2)\), \(\widetilde H^{\rm fail}(n_1,n_2)\), and \(H^{\rm fail}(n_1,n_2)\) lying over a given connected component of \(H(n_1,n_2)\).
In the second step, we understand the connected components of \(H(n_1,n_2)\).

Choose the following base point \(B=(B_1,B_2) \in U(n_1,n_2)_{\CC}\):
identify the \(\CC\)-points of \(\PP^1-\{\infty\}\) with the points of the complex plane \(\CC\);
our \(B_1\) consists of the points \((-n_1+\mathbf{i}), \dots, (-1+\mathbf{i})\) and \(B_2\) consists of the points \((1+\mathbf{i}), \dots, (n_{2}+\mathbf{i})\) (see \Cref{fig:diagram} for a rough sketch).
An \(A_4\)-quartic cover of \(\PP^1_{\CC}\) branched over \((B_1,B_2)\) is defined by the monodromy map
\[ \mu \colon \pi_1(\PP^1_{\CC} - (B_1 \cup B_2), 0) \to S_4.\]
We choose standard generators \((\alpha_1, \dots, \alpha_{n_1}, \beta_{1}, \dots, \beta_{n_2})\) for \(\pi_1(\PP^1_{\CC} - (B_1 \cup B_2), 0)\) as shown in \Cref{fig:diagram}. 
The only relation they satisfy is 
\[ \alpha_1 \cdots \alpha_{n_1} \cdot \beta_1 \cdots \beta_{n_2} = 1.\]
(Our convention is that the product above is represented by the loop that follows \(\alpha_1\), then \(\alpha_2\), and so on.  Then the product is homotopic to a giant clockwise loop based at \(0\).)
The monodromy map \(\mu\) is equivalent to an \((n_1+n_2)\)-tuple of permutations in \(S_4\), say \((\sigma_1, \dots, \sigma_{n_1},\tau_1, \dots, \tau_{n_2})\) such that
\[ \sigma_1\cdots \sigma_{n_1} \cdot \tau_1 \cdots \tau_{n_2} = \id.\]
Since the cover has ramification of type \((3,1)\) over points of \(B_1\), each \(\sigma_i\) is a \(3\)-cycle, and similarly, each \(\tau_j\) is a \((2,2)\)-cycle (a product of two disjoint \(2\)-cycles).
Two such tuples \((\sigma_1, \dots, \sigma_{n_1},\tau_1, \dots, \tau_{n_2})\) and  \((\sigma_1', \dots, \sigma_{n_1}',\tau_1', \dots, \tau_{n_2}')\)
define isomorphic \(A_4\)-quartic covers of $\PP^1_\CC$ branched over $B$ if and only if there exists $g\in S_4$ such that $g\sigma_ig^{-1}=\sigma_i'$ and $g\tau_jg^{-1}=\tau_j'$ for all $i\in\{1,\dots, n_1\}$ and all $j\in\{1,\dots, n_2\}$.

A point of \(H^{\rm fail}(n_1,n_2)\) over \(B \in U(n_1,n_2)_{\CC}\) consists of an \(A_4\)-quartic cover of $\PP^1_\CC$ branched over $B$ described by the monodromy \(\mu\) as above, together with a labelling by \(\{\pm 1\}\) of the two points of the quartic cover over every point of \(B_2\).
Consider the \(j\)-th point of \(B_2\), around which the monodromy \(\tau_j\) is the product of two disjoint \(2\)-cycles \((pq)(rs)\).
The two points of the \(A_4\)-quartic cover over this point correspond naturally to the 2-cycles \((pq)\) and \((rs)\).
Labelling these two points is equivalent to labelling these 2-cycles.
We indicate the labelling of the cycles by a decoration---we underline the 2-cycle labelled by \(+1\).
We use \(\hat{\tau}_j\) to refer to a \((2,2)\)-cycle \(\tau_j\) together with the underlining of one of its \(2\)-cycles.

The conjugation action of \(S_4\) on the set of \((2,2)\)-cycles lifts to an action on the set of decorated \((2,2)\)-cycles as follows.
We set
\[ g \cdot \underline{(pq)}(rs) \cdot g^{-1} = \underline{g(pq)g^{-1}} \cdot {g(rs)g^{-1}}.\]
Then the points of \(H^{\rm fail}(n_1,n_2)\) over \(B\) are given by decorated monodromy data \[\hat{\mu} = (\sigma_1, \dots, \sigma_{n_1}, \hat{\tau}_1, \dots, \hat{\tau}_{n_2}),\] modulo simultaneous conjugation by \(S_4\).  
The map \(H^{\rm fail}(n_1,n_2) \to H(n_1,n_2)\) simply forgets the decoration.

The fundamental group \(\pi_1(U(n_1,n_2)_{\CC}, B)\) acts on the \(\CC\)-points of the fibres over \(B\) of the maps \(H(n_1,n_2)_{\CC} \to U(n_1,n_2)_{\CC}\) and \(H^{\rm fail}(n_1,n_2)_{\CC} \to U(n_1,n_2)_{\CC}\).
These actions are given by well-known Hurwitz moves; see for example \cite[\S1.5]{ful:69}.
We recall some important Hurwitz moves for the convenience of the reader.
An anti-clockwise half-twist that exchanges the \(i\)th and \((i+1)\)th point of \(B_1\) defines a closed loop in $U(n_1,n_2)_\CC$ based at $B$.
The action of this loop on the points of \(H^{\rm fail}(n_1,n_2)_{\CC}\) over \(B\) is given in terms of decorated monodromy by
\begin{equation}\label{eqn:halftwist1}
  (\sigma_1,\dots, \sigma_{n_1}, \hat{\tau}_1,\dots, \hat{\tau}_{n_2}) \mapsto (\sigma_1,\dots,\sigma_{i-1}, \sigma_{i}\sigma_{i+1}\sigma_{i}^{-1}, \sigma_i,\dots, \sigma_{n_1}, \hat{\tau}_1,\dots, \hat{\tau}_{n_2}).
\end{equation}
Similarly, an anti-clockwise half-twist that exchanges the \(j\)th and \((j+1)\)th points of \(B_2\) acts by
\begin{equation}\label{eqn:halftwist2}
  (\sigma_1,\dots, \sigma_{n_1}, \hat{\tau}_1,\dots, \hat{\tau}_{n_2}) \mapsto (\sigma_1,\dots, \sigma_{n_1}, \hat{\tau}_1,\dots,\hat{\tau}_{j-1} ,\tau_{j}\hat{\tau}_{j+1}\tau_{j}^{-1}, \hat{\tau}_j,\dots, \hat{\tau}_{n_2}).
\end{equation}
An anti-clockwise half-twist exchanging the last point of $B_1$ and the first point of $B_2$ is not a closed loop in $U(n_1,n_2)_{\CC}$. However, the anti-clockwise full twist is a closed loop; this full twist acts by
\begin{equation}\label{eqn:fulltwist}
  (\sigma_1,\dots, \sigma_{n_1}, \hat{\tau}_1,\dots, \hat{\tau}_{n_2}) \mapsto
(\sigma_1,\dots, \sigma_{n_1-1}, (\sigma_{n_1}\tau_1)^{2}\sigma^{-1}_{n_1},\sigma_{n_1}\hat{\tau}_1 \sigma^{-1}_{n_1}, \hat{\tau}_2,\dots, \hat{\tau}_{n_2}).
\end{equation}
The action on the fibres of \(H(n_1,n_2)_{\CC} \to U(n_1,n_2)_{\CC}\) is given by the same formulas, but without the decoration.

A point of \(\widetilde H(n_1,n_2)_{\CC}\) over \(B \in U(n_1,n_2)_{\CC}\) consists of a point of \(H(n_1,n_2)_{\CC}\) together with a numbering of the points of \(B_2\).
Likewise, a point of \(\widetilde H^{\rm fail}(n_1,n_2)_{\CC}\) over \(B \in U(n_1,n_2)_{\CC}\) consists of a point of \(H^{\rm fail}(n_1,n_2)_{\CC}\) together with a numbering of the points of \(B_2\).
The action of \(\pi_1(U(n_1,n_2)_{\CC}, B)\) on the numbering is via the natural homomorphism \(\pi_1(U(n_1,n_2)_{\CC},B) \to S_{n_2}\).
We summarise the discussion above in \Cref{tab:points}.

\begin{table}
  \centering
  \begin{tabular}{l p{.8\textwidth}} \toprule
      Space & Fibre over $B=(B_1, B_2)$  
      \\ \midrule    
    \(H(n_1,n_2)\) & Tuples \((\sigma_1, \dots, \sigma_{n_1}, \tau_1, \dots, \tau_{n_2})\) where \(\sigma_i \in S_4\) is a 3-cycle and \(\tau_j \in S_4\) is a \((2,2)\)-cycle with \(\sigma_1 \dots \sigma_{n_1}\cdot \tau_1\dots  \tau_{n_2} = 1\), up to conjugation by \(S_4\).\\
    \(H^{\rm fail}(n_1,n_2)\) & Tuples as above together with the underlining of a 2-cycle in each \(\tau_j\), denoted by \((\sigma_1, \dots, \sigma_{n_1}, \hat{\tau}_1, \dots, \hat{\tau}_{n_2})\), up to conjugation by \(S_4\). \\
    \(\widetilde H(n_1,n_2)\) & Tuples \((\sigma_1, \dots, \sigma_{n_1}, \tau_1, \dots, \tau_{n_2})\) as above together with a numbering of the points of \(B_2\), up to conjugation by \(S_4\).\\
    \(\widetilde H^{\rm fail}(n_1,n_2)\) & Tuples \((\sigma_1, \dots, \sigma_{n_1}, \hat{\tau}_1, \dots, \hat{\tau}_{n_2})\) as above together with a numbering of the points of \(B_2\), up to conjugation by \(S_4\). \\
    \bottomrule
    \end{tabular}
    \caption{A description of the fibres of the Hurwitz spaces over the fixed base point \(B \in U(n_1,n_2)_{\CC}\).}
    \label{tab:points}
  \end{table}

We begin with a useful lemma.
\begin{lemma}\label{lem:color-change}
  Fix the basepoint \(B \in U(n_1,n_2)_{\CC}\) as above.
  Let \(H \subset H(n_1,n_2)_{\CC}\) be a connected component.
  There exists a point of \(H\) over $B$ whose monodromy \((\sigma_1, \dots, \sigma_{n_1}, \tau_1, \dots, \tau_{n_2})\) has \(\tau_i = \tau_j\) for all \(i,j \in \{1, \dots, n_2\}\).
\end{lemma}
\begin{proof}
  Start with any point of \(H\) over \(B\) given by a monodromy \(\mu = (\sigma_1, \dots, \sigma_{n_1}, \tau_1, \dots, \tau_{n_2})\).
  We modify it so that all \(\tau_i\)'s become equal.

  Consider the anti-clockwise full twist in \(U(n_1,n_2)_{\CC}\) in the last point of \(B_1\) and the first point of \(B_2\).
  Acting by this loop on \(\mu\) yields a \(\mu'\) that differs from \(\mu\) only at the positions of \(\sigma_{n_1}\) and \(\tau_{1}\).
  Using \eqref{eqn:fulltwist}, we compute the new \(\sigma_{n_1}\) and \(\tau_1\) after zero, one, and two full twists, starting with \(\sigma_{n_1} = (123)\) and \(\tau_1 = (12)(34)\) (our convention for multiplying permutations is that \(p \cdot q = p \circ q\)).
  \begin{center}
    \begin{tabular}{l c c } \toprule
          & \(\sigma_{n_1}\) & \(\tau_1\) \\ \midrule
        Initial & \((123)\) & \((12)(34)\) \\
      After one twist& \((243)\) & \((14)(23)\)\\
      After two twists& \((142)\) & \((13)(24)\)\\ \bottomrule
    \end{tabular}
  \end{center}
   We see that by zero, one, or two full twists, we can change the \((2,2)\)-cycle \(\tau_1\) to any \((2,2)\)-cycle.

  Since the \((2,2)\)-cycles commute with each other, a half-twist exchanging two adjacent points of \(B_2\) simply exchanges the monodromy elements.

  So, if we have \(\tau_1 = \tau_2 = \cdots = \tau_j \neq \tau_{j+1}\), we do a sequence of half-twists to bring the \((j+1)\)-th point to the first position; do one or more full twists around the last point of \(B_1\) to make the monodromy around this point equal to \(\tau_j\),  and do the reverse sequence of half-twists to bring it back to the \((j+1)\)th position.
  By repeating this procedure, we make \(\tau_1 = \cdots = \tau_{n_2}\).  
\end{proof}

\begin{proposition}\label{prop:comp-Htilde}
  Let \(H \subset H(n_1,n_2)_{\CC}\) be a connected component.
  The pre-image of \(H\) in \(\widetilde H(n_1,n_2)_{\CC}\) is connected.
\end{proposition}
\begin{proof}
  Recall that the points of \(\widetilde H(n_1,n_2)_{\CC}\) are the points of \(H(n_1,n_2)_{\CC}\) together with a numbering of the points of \(B_2\).
  Choose a point \(f \in H\) over the fixed basepoint \(B\) as above, whose monodromy is \((\sigma_1, \dots, \sigma_{n_1}, \tau_1, \dots, \tau_{n_2})\) with \(\tau_1 = \dots = \tau_{n_2}\).
  This is possible by \Cref{lem:color-change}.
  The pre-image of \(f\) in \(\widetilde H(n_1,n_2)_{\CC}\) consists of copies of $f$ 
  equipped with different numberings of the \(n_2\) points of \(B_2\).
  We will show that the action of \(\pi_1(H, f)\) on the points of the fibre of \(\widetilde H(n_1,n_2)_{\CC}\) over \(f\) is transitive.
  By the path lifting property for covering spaces, this will imply path-connectedness, and hence connectedness, of the pre-image of $H$.

  Consider the half-twist that interchanges two adjacent points of \(B_2\).
  Since the monodromy around all points of \(B_2\) is the same, this half-twist does not change the monodromy, and thus lifts to a loop in \(H\).
  The action of this loop on the pre-image of $f$ simply exchanges the numbers associated with the two interchanged points of \(B_2\).
  By a sequence of pairwise interchanges, we may change a given numbering of the points of \(B_2\) to any other numbering.
\end{proof}

\begin{proposition}\label{prop:comp-Hfailtilde}
  Let \(H \subset H(n_1,n_2)_{\CC}\) be a connected component.
  The pre-image of \(H\) in \(\widetilde H^{\rm fail}(n_1,n_2)_{\CC}\) has at most two connected components.
   Furthermore, the action of \(S_{n_2}\) on this preimage preserves the connected components. 
 \end{proposition}
\begin{proof}
  Let \(\widetilde H\) be the pre-image of \(H\) in \(\widetilde H(n_1,n_2)_{\CC}\).
  By \Cref{prop:comp-Htilde}, \(\widetilde H\) is connected.
  Choose a point \(\widetilde f \in \widetilde{H}\) given by the monodromy \((\sigma_1,\dots, \sigma_{n_1}, \tau_1, \dots, \tau_{n_2})\) with the left-to-right numbering of the points of \(B_2\). 
  The fibre of \(\widetilde f\) in \(\widetilde H^{\rm fail}(n_1,n_2)_{\CC}\) consists of the \(2^{n_2}\) possible decorations of \(\tau_1, \dots, \tau_{n_2}\).
  We show that the action of \(\pi_1(\widetilde H, \widetilde f)\) on this set has at most two orbits.

  By an argument similar to the proof of \Cref{lem:color-change}, we may assume that \(\tau_2 = \dots = \tau_{n_2}\) and \(\tau_1 \neq \tau_2\).
  Without loss of generality, take \(\tau_1 = (12)(34)\) and \(\tau_2 = \dots = \tau_{n_2} = (13)(24)\).
  Consider a decoration \((\hat{\tau}_1, \dots, \hat{\tau}_{n_2})\) of \((\tau_1, \dots, \tau_{n_2})\).
Given \(j \in \{2, \dots, n_2\}\), we consider the loop in \(U(n_1,n_2)_{\CC}\) that moves point \(j\) up, then left towards point \(1\) of $B_2$, takes it in a full twist around point \(1\), and then takes it back to its original position along the same path (see below): 
\[
  \begin{tikzpicture}[scale=0.5]
    \draw[fill] (-2,0) circle (0.1)  (-1,0) circle (0.1) (0,0) node {\(\cdots\)} node[below=0.2cm] {\(B_2\)} (1,0) circle (0.1) (2,0) node {\(\cdots\)} (3,0) circle (0.1);
    \draw[dashed] (1,0) -- (1,1) -- (-2,1);
    \draw [dashed, postaction={decorate}, decoration={markings, mark=at position 0.5 with \arrow{To}}]
    plot[smooth] coordinates {(-2,1) (-2.4,0) (-2,-0.4) (-1.6,0) (-2,1)};
  \end{tikzpicture}
\]
This loop has no effect on the monodromy or on the numbering of \(B_2\).
So it lifts to a closed loop in \(\widetilde H\) based at \(\widetilde f\).
Let us analyse its effect on the decorated monodromy.
It is easy to see that the only possible change is in \(\hat{\tau}_1\) and \(\hat{\tau}_j\).
In particular, we see that \(\hat{\tau}_j\) changes to the conjugate \(\tau_1 \hat{\tau}_j \tau_1^{-1}\).
We have
\[ (12)(34) \cdot {(13)}\underline{(24)} \cdot (12)(34) = \underline{(13)}{(24)};\]
that is, the conjugate \(\tau_1 \hat{\tau}_j \tau_1^{-1}\) is the same \((2,2)\)-cycle as \(\hat{\tau}_j\) but with the other underlining.
By a sequence of such loops, we can ensure that the decorated monodromy around the points \(2, \dots, n_2\) is \(\underline{(13)}{(24)}\).
We conclude that under the action of \(\pi_1(\widetilde H, \widetilde f)\), any decoration \((\hat{\tau}_1, \dots, \hat{\tau}_{n_2})\) lies in the orbit of 
\[ (\underline{(12)}{(34)}, \underline{(13)}{(24)}, \dots, \underline{(13)}{(24)}) \text{ or } ((12)\underline{(34)}, \underline{(13)}{(24)}, \dots, \underline{(13)}{(24)}).\]
It follows that the pre-image of \(\widetilde H\) has at most two components.

We now prove that \(S_{n_2}\) preserves the connected components.
Let \(f \in H(n_1,n_2)_{\CC}\).
It suffices to prove that for all \(\hat{f}\) in \(\widetilde H^{\rm fail}(n_1,n_2)_{\CC}\) over \(f\) and \(p \in S_{n_2}\), the points \(\hat{f}\) and \(p \cdot \hat{f}\) are in the same connected component of \(\widetilde H^{\rm fail}(n_1,n_2)_{\CC}\).
Since \(\widetilde H^{\rm fail}(n_1,n_2)_{\CC}\) has at most two connected components over the connected component of \(H(n_1,n_2)_{\CC}\) containing \(f\), it suffices to prove this for any one \(\hat{f}\).
We choose \(f\) branched over \(B_1\cup B_2\) with the monodromy \((\sigma_1, \dots, \sigma_{n_1}, \tau_1, \dots, \tau_{n_2})\) satisfying \(\tau_i = \tau_j\) for all \(i,j\) (see \Cref{lem:color-change}).
We then choose \(\hat{f}\) to have the left-to-right numbering of the points of \(B_2\) and the same \(2\)-cycle underlined in all the \(\tau_i\).
Let \(p \in S_{n_2}\) be the transposition \((j,j+1)\).
A half-twist in the \(j\)-th and \((j+1)\)-th points of \(B_2\) lifts to a path in \(\widetilde H^{\rm fail}(n_1,n_2)_{\CC}\) connecting \(\hat{f}\) and \(p \cdot \hat{f}\). 
We conclude that \(\hat{f}\) and \(p \cdot \hat{f}\) lie in the same connected component.
Since the transpositions \((j,j+1)\) generate \(S_{n_2}\), the same follows for any \(p \in S_{n_2}\).
\end{proof}

\begin{proposition}\label{prop:comp-hfail}
  Let \(H \subset H(n_1,n_2)_{\CC}\) be a connected component.
  The map \(\widetilde H^{\rm fail}(n_1,n_2)_{\CC} \to H^{\rm fail}(n_1,n_2)_{\CC}\) is a bijection on the connected components over \(H\).
  In particular, \(H^{\rm fail}(n_1,n_2)_{\CC}\) also has at most two connected components over \(H\).
\end{proposition}
\begin{proof}
  Recall that \(\widetilde H^{\rm fail}(n_1,n_2)_{\CC} \to H^{\rm fail}(n_1,n_2)_{\CC}\) is the quotient by \(S_{n_2}\).
  The statement follows from \Cref{prop:comp-Hfailtilde}.
\end{proof}

Having discussed the connected components of \(\widetilde H(n_1,n_2)\), \(\widetilde H^{\rm fail}(n_1,n_2)\) and \(H^{\rm fail}(n_1,n_2)\) over a connected component of \(H(n_1,n_2)\), we now address the connected components of \(H(n_1,n_2)\).
Let \(v(n_1,n_2)\) be the number of connected components of \(H(n_1,n_2)_{\CC}\).
Set \(n = n_1+n_2\) and let \(U \subset \PP^n\) be the complement of the discriminant divisor.
Recall that we have a finite \'etale morphism \(H(n_1,n_2) \to U\).
Consider the point \(B \in U(n_1,n_2)_{\CC}\) also as a point of \(U_{\CC}\).
Let \(V(n_1,n_2)\) be the set
\begin{equation}\label{def:V} 
  V(n_1,n_2) = \left\{(g_1, \dots, g_n) \Biggm|
    \begin{split}
      &g_i \in S_4, \quad g_1 \cdots g_n = 1,\\
      &\text{\(n_1\) of the \(g_i\) have cycle type \((3,1)\),} \\
      &\text{\(n_2\) of the \(g_i\) have cycle type \((2,2)\).}
    \end{split}
    \right\}.
  \end{equation}  
  The group \(S_4\) acts on \(V(n_1,n_2)\) by simultaneous conjugation
\[ h \colon (g_1, \dots, g_n) \mapsto (hg_1h^{-1}, \dots,  hg_nh^{-1}). \] 
 
   The points of \(H(n_1,n_2)_{\CC}\) over \(B \in U_{\CC}\) are in natural bijection with \(V(n_1,n_2)\) modulo the conjugation action of \(S_4\).

   The fundamental group \(\pi_1(U_{\CC},B)\) is the braid group \(\mathbb{B}_n\).
The \(i\)th generator \(s_i\) represents the anti-clockwise half-twist that interchanges the \(i\)th and \((i+1)\)th points of \(B\).
The group
\(\mathbb{B}_n\) acts on \(V(n_1,n_2)\) by 
\[s_i \colon (g_1, \dots, g_n) \mapsto (g_1, \dots,  g_{i-1}, g_{i}g_{i+1}g_{i}^{-1}, g_{i}, g_{i+2},\dots, g_n).\]
The connected components of \(H(n_1,n_2)_{\CC}\) are in natural bijection with the orbits of the \(\mathbb{B}_n \times S_4\) action on \(V(n_1,n_2)\) \cite[Section 1.3]{fri.vol:91}. In particular, \(v(n_1,n_2)\) is the number of orbits.
This is bounded above by the number of {\(\mathbb{B}_n\)}-orbits.
We analyse the latter using the lifting invariant and the Conway--Parker theorem, which we first recall from \cite{woo:21}.

Let \(G\) be a finite group and let \(c \subset G\) be a subset that generates \(G\), does not contain the identity, and is closed under conjugation by \(G\).
Let \(U(G,c)\) be the group with generators \([g]\) for every \(g \in c\) and relations
\[ [x][y][x]^{-1} = [xyx^{-1}],\]
for every \(x,y \in c\). 
We have a homomorphism \(U(G,c) \to G\) that sends the generator \([g]\) to \(g\).
Let \(D\) be the set of conjugacy classes in \(c\).
We have a homomorphism \(U(G,c) \to \ZZ^D\) that sends the generator \([g]\) to the conjugacy class of \(g\).
We also have a homomorphism \(\ZZ^D \to G^{\rm ab}\) that sends the conjugacy class of \(g\) to the image of \(g\) in \(G^{\rm ab}\).
The homomorphisms above combine to give a homomorphism
\[ U(G,c) \to G \times_{G^{\rm ab}} \ZZ^D.\]

Given a positive integer \(n\), set
\[ V_n^G = \{(g_1, \dots, g_n) \mid g_i \in c,\ g_1\dots g_n=1 \text{ and } g_1,\dots , g_n \text{ generate } G\}.\]
We have a map \[\Pi \colon V_n^G \to \ker(U(G,c)\to G)\] that sends \((g_1,\dots, g_n)\) to \([g_1]\cdots[g_n]\).
This map is constant on {\(\mathbb{B}_n\)} orbits and hence descends to a map \(V_n^G/{\mathbb{B}_n} \to \ker(U(G,c)\to G)\).
By composition, we have a map \(V_n^G/{\mathbb{B}_n} \to \ZZ^D\).
Let \( V_{\underline{m}}\subset \bigcup_{n\geq 0}V_n^G/{\mathbb{B}_n} \) and \(U(G,c)_{\underline{m}}\subset U(G,c)\) be the pre-images of \(\underline{m}\in\ZZ^D\). 
\begin{theorem}[Conway--Parker]\label{thm:CP}
  There exists a constant $M\in\ZZ_{>0}$ such that if all entries of $\underline{m}\in\ZZ^D$ are at least $M$, the map \(\Pi\) induces a bijection
  \[ V_{\underline{m}} \to  U(G,c)_{\underline{m}}\cap \ker(U(G,c)\to G).\]
\end{theorem}
\begin{proof}
See~\cite[Theorem~12.5]{liu.woo.zur:24},~\cite[Theorem~3.1]{woo:21} and its proof.
\end{proof}

We take \(G = A_4\) and \(c = A_4- \{\id\}\).
Then \(c\) is the union of 3 conjugacy classes, represented by \((123)\), \((132)\), and \((12)(34)\).
The commutator subgroup of \(G\) is the Klein 4-group \(\{\id, (12)(34), (13)(24), (14)(23)\}\) and \(G^{\rm ab} \cong C_3\).
Choosing the image of \((123)\) as a generator fixes an isomorphism \(G^{\rm ab} \to C_3 = \ZZ/3\ZZ\).
Then the map \(\ZZ^D = \ZZ^3 \to C_3\) sends the basis vector \(e_1\) to \(1\), \(e_2\) to \(-1\), and \(e_3\) to \(0\).
Write \(U(A_4) = U(A_4,c)\). 
\begin{proposition}\label{prop:UGc}
  The natural homomorphism \(U(A_4) \to A_4 \times_{C_3} \ZZ^3\) is an isomorphism.
\end{proposition}
\begin{proof}
 In~\cite[Theorem~2.5]{woo:21}, Wood shows that $U(A_4)$ is isomorphic to $S_c\times_{C_3}\ZZ^3$, where $S_c$ is the reduced Schur cover defined before~\cite[Lemma~2.2]{woo:21}. As in the proof of~\cite[Proposition~8.11]{ell.tra.wes:}, viewing $A_4$ as $\PSL_2(\FF_3)$, one sees that a Schur cover for $A_4$ is $\SL_2(\FF_3)$ and deduces that $S_c=A_4$. The proof of~\cite[Theorem~2.5]{woo:21} shows that the natural homomorphism \(U(A_4) \to A_4 \times_{C_3} \ZZ^3\) is an isomorphism. 
\end{proof}
Observe that a \(3\)-cycle and a \((2,2)\)-cycle are enough to generate \(A_4\).
  So, as long as \(n_1, n_2 > 0\), for every \((g_1,\dots, g_n) \in V(n_1,n_2)\), the \(g_i\) generate \(A_4\). 
  That is, we have the inclusion \(V(n_1,n_2) \subset V^{A_4}_n\).
Given non-negative integers \(a,b\) with \(a+b = n_1\), let \(V(a,b,n_2) \subset V(n_1,n_2)\) be the pre-image of \((a,b,n_2) \in \ZZ^3\) under the composition $V_n^{A_4}\xrightarrow{\Pi} U(A_4)\twoheadrightarrow \ZZ^3$.
Then \(V(n_1,n_2)\) is the disjoint union
\[ V(n_1,n_2) = \bigsqcup_{a+b=n_1} V(a,b,n_2).\]
The \({\mathbb{B}_n}\)-action on \(V(n_1,n_2)\) preserves \(V(a,b,n_2)\).
The \(S_4\)-action preserves the union
\begin{equation}\label{def:Vab}
  V(\{a,b\}, n_2) = V(a,b,n_2) \cup V(b,a,n_2).
\end{equation}
The even permutations preserve the two sets individually, but the odd permutations interchange them.

\begin{proposition}\label{prop:cpa4}
  \begin{enumerate}
  \item There exists $\lambda\in\mathbb{Z}_{>0}$ such that for all \(a,b,n_2,n \in \ZZ_{\geq 0}\) with $a+b+n_2=n$, 
    \[ |V(a,b,n_2)/{\mathbb{B}_n}| \leq \lambda.\]
   \item There exists $M\in\mathbb{Z}_{>0}$ such that for all \(a,b,n_2, n \in \ZZ_{\geq 0}\) with \(a,b,n_2 \geq M\) and $a+b+n_2=n$,
    \[ |V(a,b,n_2)/{\mathbb{B}_n}| = \begin{cases} 1 & \text{ if } a \equiv b \pmod 3 \\
    0 &\text{ otherwise.}
  \end{cases}\]
  \end{enumerate}
\end{proposition}
\begin{proof}
  The first assertion follows from 
  \cite[Lemma~3.3]{ell.ven:05}.
  We prove the second assertion. Let $n_1=a+b$.
  Recall that, as long as \(n_1, n_2 > 0\), we have \(V(a,b,n_2)\subset V(n_1,n_2) \subset V^{A_4}_n\).
  For \(a,b,n_2 \geq M\), the Conway--Parker theorem (\Cref{thm:CP}) gives a bijection
 \[V(a,b,n_2)/B_n \to U(A_4)_{(a,b,n_2)}\cap \ker(U(A_4) \to A_4).\] 
  The isomorphism \(U(A_4) \cong A_4 \times_{C_3} \ZZ^3 \) from \Cref{prop:UGc} gives
  \[ \ker(U(A_4) \to A_4) = \{(a,b,c) \in \ZZ^3 \mid a \equiv b \pmod 3\}.\]
  The statement follows.  
\end{proof}
\Cref{prop:cpa4} implies that certain geometric components of \(H(n_1,n_2) \otimes \FF_p\) are defined over \(\FF_p\), as we now explain.
Recall that \({\mathbb{B}_n} \times S_4\) orbits of \(V(n_1,n_2)\) correspond to connected components of \(H(n_1,n_2)_{\CC}\).
Given non-negative integers \(a \leq b\), the subset \(V(\{a,b\}, n_2) \subset V(n_1,n_2)\) is a union of \({\mathbb{B}_n} \times S_4\) orbits and therefore corresponds to a union of connected components of \(H(n_1,n_2)_{\CC}\).

\begin{proposition}\label{prop:components-defined-over-Q}
  The union of the connected components of \(H(n_1,n_2)_{\CC}\) corresponding to \(V(\{a,b\}, n_2)\) is defined over \(\QQ\).
  In particular, if \(V(\{a,b\}, n_2)\) forms one orbit under \({\mathbb{B}_n} \times S_4\), then the unique connected component of \(H(n_1,n_2)_{\CC}\) corresponding to this orbit is defined over \(\QQ\).
\end{proposition}

\begin{proof}
  We may replace \(\CC\) by \(\overline \QQ\).
  Let \(H \subset H(n_1,n_2)_{\overline\QQ}\) be the union of connected components corresponding to \(V(\{a,b\},n_2)\).
  Let us prove that \(H\) is defined over \(\QQ\).
  Consider a \(\overline \QQ\)-point of \(H\) corresponding to the isomorphism class of an \(A_4\)-quartic cover \((C, B_1, B_2, C \xrightarrow{f} \PP^1_{\overline \QQ})\).
  Choose a point \(0 \in \PP^1_{\overline \QQ}\) away from \(B_1\cup B_2\); fix a numbering of \(f^{-1}(0)\); and let
  \[ \mu \colon \pi_1(\PP^1_{\CC} -(B_1\cup B_2)
  , 0) \to S_4\]
  be the monodromy of \(f\).
  Recall that the image of \(\mu\) is \(A_4 \subset S_4\).

  For each \(x \in B_1\), let \(D_x \subset \PP^1_{\CC}\) be a small disc centred at \(x\), and \(0_x \in D_x\) a point other than \(x\).
  Let \(\gamma_x \in \pi_1(\PP^1_{\CC}-(B_1\cup B_2), 0_x)\) be the element represented by an anti-clockwise loop in \(D_x\) around \(x\) based at \(0_x\).
  Recall that a path from \(0\) to \(0_x\) in \(\PP^1_{\CC}-(B_1\cup B_2)\) gives an isomorphism
  \[\pi_1(\PP^1_{\CC}-(B_1\cup B_2), 0_x) \xrightarrow{\sim} \pi_1(\PP^1_{\CC}-(B_1\cup B_2), 0);\]
  two isomorphisms given by different paths are conjugate to each other.
  Thus, the conjugacy class of the image of \(\gamma_x\) in \(\pi_1(\PP^1_{\CC}-(B_1\cup B_2), 0)\) is well defined independently of the path.
  Denote it by \([\gamma_x] \subset \pi_1(\PP^1_{\CC}-(B_1\cup B_2), 0)\).
  Then \([\mu(\gamma_x)] \subset A_4\) is a well-defined conjugacy class of \(A_4\).
  Let \(M(f)\) be the multiset
  \[M(f) = \{[\mu(\gamma_x)] \mid x \in B_{1}\}.\]
  Changing the numbering of \(f^{-1}(0)\) changes the multiset by simultaneous conjugation by an element of \(S_4\).
  The \(f \in H(n_1,n_2)_{\CC}\) lying in $H$ are characterised by the condition that, up to simultaneous conjugation by \(S_4\), we have
  \[ M(f) = \{a \cdot [(123)], b \cdot [(132)]\}.\]
    
  Take \(\varphi \in \Gal(\overline \QQ/\QQ)\) and let \(\chi \in \varprojlim (\ZZ/m\ZZ)^{\times}\) be its cyclotomic character. 
  Then, by \cite[\S~4.3.1]{wew:98}, we have 
  \[ M(f^{\varphi}) = \{[\mu(\gamma_x)]^{\chi} \mid x \in B_1\}.\]
  We see that, up to conjugation by \(S_4\), the multiset \(M(f^{\varphi})\) is also \(\{a \cdot [(123)], b \cdot [(132)]\}\).
  So \(f^{\varphi}\) also lies in \(H\).
  Thus, \(H\) is invariant under \(\operatorname{Gal}(\overline \QQ/\QQ)\) and hence is defined over \(\QQ\).
\end{proof}

\Cref{prop:components-defined-over-Q} enables us to define \(H(\{a,b\}, n_2) \subset H(n_1,n_2)\) as the open and closed subspace such that \(H(\{a,b\},n_2)_{\CC}\) is the union of the connected components of \(H(n_1,n_2)_{\CC}\) corresponding to \(V(\{a,b\}, n_2)\). 

We now have all the tools to apply the Chebotarev density theorem.
Recall from \Cref{prop:sterlingapprox} that \(b_\ell = \sum_{i=1}^\ell2^{-i}s(\ell,i)/\ell!\).
\begin{theorem}\label{thm:failbound}
  Let \(n_1 > 0\) and \(n_{2} \geq 0\) be integers and \(p \geq 5\) a prime.
  Let \(q\) be a power of \(p\) such that all geometric components of \(\widetilde H^{\rm fail}(n_1,n_2) \otimes \FF_p\) are defined over \(\FF_q\).
  Let \(\phi \colon  H^{\rm fail}(n_1,n_2) \to H(n_{1},n_2)\) be the natural forgetful map. 
Let \(H \subset H(n_1,n_2) \otimes \FF_q\) be a (geometric) component and let \(H^{\rm fail}\) be the pre-image of \(H\) in \(H^{\rm fail}(n_1,n_2) \otimes \FF_q\).
  Then,
   \[ \lim_{m \to \infty} \frac{|\phi(H^{\rm fail}(\FF_{q^m}))|}{|H(\FF_{q^m})|} = rb_{n_2},\]
  where $r\leq 2$ is the number of geometric components of \(\widetilde H^{\rm fail}(n_1,n_2) \otimes \FF_q\) lying above $H$. Consequently,
   \[ \lim_{m \to \infty} \frac{|\phi(H^{\rm fail}(n_1,n_2)(\FF_{q^m}))|}{|H(n_1,n_2)(\FF_{q^m})|} \leq 2 b_{n_2}.\]
  
\end{theorem}
\begin{proof}
  Let 
  \(\widetilde H^{\rm fail}\) be the pre-image of \(H\) in 
  \(\widetilde H^{\rm fail}(n_1,n_2) \otimes \FF_q\). 
   Let \(G = (\ZZ/2\ZZ)^{n_2} \rtimes S_{n_2}\).
  We know from \eqref{eqn:fourspaces} that the map \(\widetilde H^{\rm fail} \to H\) is a \(G\)-torsor and \(\widetilde H^{\rm fail} \to H^{\rm fail}\) is the quotient by \(S_{n_2} \subset G\). 
  By \Cref{prop:sterling}, the union of the conjugates of \(S_{n_2} \subset G\) has size \(b_{n_2} \cdot |G|\).
  
We will apply the Chebotarev density theorem (\Cref{prop:chebapply}) with $G=(\ZZ/2\ZZ)^{n_2}\rtimes S_{n_2}$ and \(\Phi = \bigcup_{h \in G} h S_{n_2} h^{-1}\). First, we show that $\Phi$ is contained in the stabiliser $G_0$ of a geometric component of $\widetilde H^{\rm fail}$. By Propositions~\ref{prop:components-complex} and~\ref{prop:comp-Hfailtilde}, $[G:G_0]\leq 2$ and $S_{n_2}\subset G_0$. Therefore, $G_0$ is a normal subgroup of $G$ containing $S_{n_2}$, which implies that $\Phi\subset G_0$, as claimed.

 Now, by~\Cref{prop:chebapply}, we get
  \[ \lim_{m \to \infty} \frac{|\phi(H^{\rm fail}(\FF_{q^m}))|}{|H(\FF_{q^m})|} = rb_{n_2},\]
  with $r\leq 2$, thanks to Propositions~\ref{prop:components-complex} and~\ref{prop:comp-Hfailtilde}.
  Summing over all geometric components 
 of \(H(n_1,n_2) \otimes \FF_q\) yields the second statement. 
\end{proof}
\begin{remark} Let \(k\) be a finite field of characteristic \(\neq 2,3\).
  By \Cref{thm:failset}, the image of the $k$-points of \(H^{\rm fail}(n_1,n_2)\) are precisely the isomorphism classes of \(A_4\)-quartic covers of $\PP^1_k$ that fail the Hasse norm principle.
\Cref{thm:failbound} says that, asymptotically in \(m\), the proportion of isomorphism classes of \(A_{4}\)-quartic covers of \(\PP^1_{\FF_{q^m}}\) that fail the Hasse norm principle is bounded above by \(2b_{n_2}\).
By \Cref{prop:sterlingapprox}, \(b_{n_2}\) is within a constant factor of \(n_2^{-1/2}\).
\end{remark}

\begin{corollary}\label{cor:failbound-limit}
  In the set-up of \Cref{thm:failbound}, we have
  \[ \lim_{n_2 \to \infty }\lim_{m \to \infty} \frac{|\phi(H^{\rm fail}(n_1,n_2)(\FF_{q^m}))|}{|H(n_1,n_2)(\FF_{q^m})|} = 0.\]
\end{corollary}
\begin{proof}
  Use \(b_{n_2} \sim n_2^{-1/2}\) from \Cref{prop:sterlingapprox}.
\end{proof}

Recall that \(H(\{a,b\},n_2) \subset H(n_1,n_2)\) is the open and closed subspace corresponding to the set \(V(\{a,b\},n_2)\) (see \Cref{prop:components-defined-over-Q}).
Let \(H^{\rm fail}(\{a,b\},n_2) \subset H^{\rm fail}(n_1,n_2)\) be the pre-image of \(H(\{a,b\},n_2)\).
Under the one-orbit
hypothesis of \Cref{prop:precise-q} below, we can take $q=p$ in \Cref{thm:failbound}. 
By \Cref{prop:cpa4}, the one-orbit hypothesis holds if \(a\), \(b\) and \(n_2\) are sufficiently large.
\begin{proposition}\label{prop:precise-q}
  Suppose \(a,b,n_2\) are such that \(V(\{a,b\},n_2)\) forms one orbit under \({\mathbb{B}_n} \times S_4\).
  Then \(H(\{a,b\}, n_2) \otimes \FF_p\) is geometrically connected and
  \[ \limsup_{m \to \infty} \frac{|\phi(H^{\rm fail}(\{a,b\},n_2)(\FF_{p^m}))|}{|H(\{a,b\},n_2)(\FF_{p^m})|} =\lim_{m \to \infty} \frac{|\phi(H^{\rm fail}(\{a,b\},n_2)(\FF_{p^{2m}}))|}{|H(\{a,b\},n_2)(\FF_{p^{2m}})|}= r b_{n_2},\]
   where $r\leq 2$ is the number of geometric components of \(\widetilde H^{\rm fail}(n_1,n_2) \otimes \FF_p\) lying above \(H(\{a,b\}, n_2) \otimes \FF_p\). 
\end{proposition}
\begin{proof}
  The first statement follows from Propositions~\ref{prop:components-complex} and~\ref{prop:components-defined-over-Q}.
  By \Cref{prop:comp-Hfailtilde}, \(\widetilde H^{\rm fail}(n_1,n_2) \otimes \FF_p\) has at most two geometric components over \(H(\{a,b\}, n_2) \otimes \FF_p\).
  If these geometric components are defined over \(\FF_p\) then we apply the Chebotarev density theorem as in the proof of \Cref{thm:failbound}.
  Now suppose they are not defined over \(\FF_p\).
  That is, there are two geometric components and they are interchanged by \(\Frob_p\).
  By \Cref{prop:comp-hfail}, \(H^{\rm fail}(\{a,b\},n_2)\) also has two geometric components and they are interchanged by \(\Frob_p\).
  Then for any odd \(m\), we have \(H^{\rm fail}(\{a,b\},n_2)(\FF_{p^m}) = 0\), so the quantity in the limit is zero.
  For even \(m\), we apply the Chebotarev density theorem as in the proof of \Cref{thm:failbound}.
\end{proof}

\section{Proof of~\Cref{thm: main result}}\label{sec: n to infinity}

When combined with \Cref{thm:failset} and \Cref{prop:sterlingapprox}, \Cref{thm:failbound} and \Cref{cor:failbound-limit} yield analogues of \Cref{thm: main result} and \Cref{Cor: main result} but with the bound in \Cref{thm: main result} replaced by $2n_2^{-\frac{1}{2}}$ and the limit $n\to\infty$ in \Cref{Cor: main result} replaced by $n_2\to \infty$. Our final task is to go from $n_2$ to $n=n_1+n_2$ and thus prove our main results: \Cref{thm: main result} and \Cref{Cor: main result}.

So, without further ado, let \(n\), \(n_1\), and \(n_2\) be non-negative integers such that \(n = n_1+n_2\) and \(n_1 > 0\).
Let \(H(n)\) be the disjoint union
\begin{equation}\label{eq:disjoint}
    H(n) =  \bigsqcup_{\substack{n_1+n_2 = n \\ n_1 > 0, n_2 \geq 0}} H(n_1,n_2).
\end{equation} 
Define \(H^{\rm fail}(n)\) analogously.
Our goal in this section is to combine the bounds in \Cref{sec:comphurwitz} to obtain a bound on 
\[ \lim_{m \to \infty} \frac{|\phi(H^{\rm fail}(n)(\FF_{p^m}))|}{|H(n)(\FF_{p^m})|}.\]

Recall that \(v(n_1,n_2)\) is the number of connected components of \(H(n_1,n_2)_{\CC}\).
Equivalently, it is the number of \({\mathbb{B}_n} \times S_4\)-orbits of the set \(V(n_1,n_2)\) defined in \eqref{def:V}.
Given non-negative integers \(a, b\) with \(a+b = n_1\), let \(v(\{a,b\},n_2)\) be the number of \({\mathbb{B}_n} \times S_4\)-orbits of the set \(V(\{a,b\},n_2)\) defined in \eqref{def:Vab}. 
We estimate \(\sum v(\{a,b\},n_2)\).
For two functions \(f,g \colon \ZZ_{\geq 0} \to \RR\), we write \(f \approx g\) if \[\lim_{n \to \infty} f(n)/g(n) = 1.\]

\begin{lemma}\label{lem:estimates}
  Fix an integer \(m\).
  We have
  \[ \sum_{\substack{a+b+n_2 = n\\ a,b,n_2 \geq m,\ a \leq b}}v(\{a,b\},n_2) \approx n^2/12.\]
  Furthermore, for \(0 < \alpha < 1\), we have
  \[ \sum_{\substack{a+b+n_2 = n\\a,b,n_2 \geq m,\ a \leq b\\n_2 < n^{\alpha}}} v(\{a,b\},n_2) \approx n^{\alpha+1}/6.\]
\end{lemma}
\begin{proof}
  From \Cref{prop:cpa4} it follows that there is a constant $\lambda$ such that
  \begin{equation}\label{eqn:constbound}
    v(\{a,b\},n_2) \leq \lambda,
  \end{equation}
  and there is a constant \(M\) such that if \(a,b, n_2 \geq M\), then
  \begin{equation}\label{eqn:modularbound}
    v(\{a,b\},n_2) = \begin{cases} 1 & \text{ if } a \equiv b \pmod 3 \\
      0 &\text{ otherwise.}\end{cases}
  \end{equation}
  We can thus break up the sum into two parts: the sum of the \(v(\{a,b\},n_2)\) where at least one of \(a,b,n_2\) is less than \(M\) and the sum of the \(v(\{a,b\},n_2)\) where all of \(a,b,n_2\) are at least \(M\).
  Since \(a+b+n_2 = n\), we see that the number of terms in the first part is a linear function in \(n\), and by \eqref{eqn:constbound}, their sum is bounded by a linear function in \(n\).
  Therefore, it suffices to show that the second part has the stated asymptotics.
  The second part is the sum   
  \begin{equation}\label{eqn:bigsum}
    \sum_{\substack{a+b+n_2 = n\\ a,b,n_2 \geq M,\ a \leq b}} v(\{a,b\},n_2).
  \end{equation}
  By \eqref{eqn:modularbound} this sum is the cardinality of the set
  \[ \{(a,b,n_2) \in \ZZ^3 \mid a \equiv b \pmod 3,\  a+b+n_2 = n,\  a \leq b,\  \text{ and } a,b,n_2 \geq M\}.\]
  It is easy to check that this cardinality is \(\approx n^2/12\).

  For the second assertion, we replace \eqref{eqn:bigsum} by
  \begin{equation}\label{eqn:bigsum2}
    \sum_{\substack{a+b+n_2 = n\\a,b,n_2 \geq M,\ a \leq b \\n_2 < n^{\alpha}}} v(\{a,b\},n_2).
  \end{equation}
  By \eqref{eqn:modularbound} this sum is the cardinality of the set
  \[ \{(a,b,n_2) \in \ZZ^3 \mid a \equiv b \pmod 3,\  a+b+n_2 = n,\  a\leq b,\  n_2 < n^{\alpha} \text{ and } a,b,n_2 \geq M\}.\]
  Again, it is easy to check that this cardinality is \(\approx n^{\alpha+1}/6.\)
\end{proof}

We now have the tools to prove the main result.
\begin{theorem}\label{thm: failbound general n}
  Let \(p \geq 5\) be a prime.
    Given \(\epsilon > 0\), if \(n\) is sufficiently large, then 
  \[ \limsup_{m \to \infty} \frac{|\phi(H^{\rm fail}(n)(\FF_{p^m}))|}{|H(n)(\FF_{p^m})|} \leq (4+\epsilon) n^{-1/3}.\]
\end{theorem}
\begin{proof}
  All the summations that follow are over \(a,b,n_2\) satisfying \(a+b+n_2 = n\), $n_2\geq 0$ and \(a+b > 0\) and \(a \leq b\).
  Under the summation signs, we drop these conditions and only write additional conditions.
  Fix \(\alpha\) to be any real number in \((0,1)\).
  Let \(M\) be an integer such that \(v(\{a,b\},n_2) = 1\) for \(a,b,n_2 \geq M\) (see \Cref{prop:cpa4}).
  Using~\eqref{eq:disjoint} and
  \[H^{\rm fail}(n) = \bigsqcup_{\substack{a+b+n_2 = n\\ n_2\geq 0 \\ a+b > 0, a \leq b}} H^{\rm fail}(\{a,b\},n_2),\]
  when $n^\alpha\geq M$, we have
  \begin{equation}\label{eqn:3parts}
    \begin{split}
      |\phi(H^{\rm fail}(n)(\FF_{p^m}))| &= \sum_{\substack{a,b \geq M\\ n_2 \geq n^{\alpha}}}|\phi(H^{\rm fail}(\{a,b\},n_2)(\FF_{p^m}))| \\
                                        &+ \sum_{\substack{a, b, n_2 \geq M\\ n_2 < n^{\alpha}}} |\phi (H^{\rm fail}(\{a,b\},n_2)(\FF_{p^m}))| \\
                                        &+ \sum_{a,b,\text{ or } n_2 < M} |\phi (H^{\rm fail}(\{a,b\},n_2)(\FF_{p^m}))|.
    \end{split}
\end{equation}
We consider the three parts of the sum separately, in sequence.

  Let us consider the first part.
  If \(n_2 \geq n^{\alpha}\) and \(a,b \geq M\), \Cref{prop:precise-q} and \Cref{prop:sterlingapprox} imply 
  \[ \limsup_{m \to \infty} \frac{|\phi(H^{\rm fail}(\{a,b\},n_2)(\FF_{p^m}))|}{|H(\{a,b\},n_2)(\FF_{p^m})|}  < 2 n^{-\alpha/2} .\]
  Therefore,    \begin{equation}\label{eqn:part1}
    \limsup_{m \to \infty} \frac{\sum_{\substack{a,b \geq M\\ n_2 \geq n^{\alpha}}}|\phi(H^{\rm fail}(\{a,b\},n_2)(\FF_{p^m}))|}{|H(n)(\FF_{p^m})|}  < 2 n^{-\alpha/2} .
  \end{equation}

  Now let us consider the second and third parts.
  We have the bound
  \[|\phi (H^{\rm fail}(\{a,b\},n_2)(\FF_{p^m}))| \leq |H^{\rm fail}(\{a,b\},n_2)(\FF_{p^m})|.\]
  By the Lang--Weil bounds, we have
  \[ \limsup_{m \to \infty} \frac{|H^{\rm fail}(\{a,b\},n_2)(\FF_{p^m})|}{p^{mn}} \leq v(\{a,b\},n_2).\]
  Therefore, for $n$ sufficiently large, we have
\begin{equation}\label{eqn:part2}
    \begin{split}
      \limsup_{m \to \infty} &\frac{1}{p^{mn}}\left(\sum_{\substack{a, b, n_2 \geq M\\ n_2 < n^{\alpha}}} |\phi (H^{\rm fail}(\{a,b\},n_2)(\FF_{p^m}))| + \sum_{a,b,\text{ or } n_2 < M} |\phi (H^{\rm fail}(\{a,b\},n_2)(\FF_{p^m}))| \right) \\
      & \leq \sum_{\substack{a, b, n_2 \geq M\\ n_2 < n^{\alpha}}} v(\{a,b\},n_2) + \sum_{a,b,\text{ or } n_2 < M} v(\{a,b\},n_2) \approx n^{\alpha+1}/6.
    \end{split}
  \end{equation}

  The last approximation follows from \Cref{lem:estimates} applied to the first sum.
  By \Cref{prop:cpa4}, the second sum is much smaller---bounded by a multiple of \(n\).

  By Propositions~\ref{prop:cpa4} and~\ref{prop:precise-q}, the components \(H(\{a,b\},n_2) \otimes \FF_p\) of \(H(n) \otimes \FF_p\) for \(a,b,n_2 \geq M\) are geometrically connected, and again by the Lang--Weil bounds, we have
  \begin{equation}\label{eqn:part3}
    \liminf_{m \to \infty} \frac{1}{p^{mn}} |H(n)(\FF_{p^m})| \geq \sum_{a,b,n_2 \geq M} v(\{a,b\},n_2) \approx n^2/12.
  \end{equation}
  The last approximation is again thanks to \Cref{lem:estimates}.

  By combining the inequalities \eqref{eqn:part1}, \eqref{eqn:part2}, and \eqref{eqn:part3}, we get
  \begin{equation}
    \begin{split}
      \limsup_{m \to \infty} \frac{|\phi(H^{\rm fail}(n)(\FF_{p^m}))|}{|H(n)(\FF_{p^m})|} &< 2n^{-\alpha/2} + \frac{\sum_{\substack{a, b, n_2 \geq M\\ n_2 < n^{\alpha}}} v(\{a,b\},n_2) + \sum_{a,b,\text{ or } n_2 < M} v(\{a,b\},n_2)}{\sum_{a,b,n_2 \geq M}v(\{a,b\},n_2)}\\
      &\approx 2n^{-\alpha/2} + 2 n^{\alpha - 1}.
    \end{split}
  \end{equation}
Taking \(\alpha = 2/3\) yields the result.
\end{proof}

\begin{remark}\label{rem:suff large}
An inspection of the proof of \Cref{thm: failbound general n} shows that determining an explicit bound beyond which $n$ is considered ``sufficiently large'' would involve determining the constants $\lambda$ and $M$ from \Cref{prop:cpa4} and then obtaining more precise estimates for the sums in \Cref{lem:estimates}. The proof of~\cite[Lemma~3.3]{ell.ven:05} shows that one can take $\lambda=|A_4|^{|A_4|^2}=12^{12^2}$. It is less clear how to make $M$ explicit.
\end{remark}

\begin{proof}[Proof of~\Cref{thm: main result}]
  Combine \Cref{thm:failset}, (\ref{eq:disjoint}) and \Cref{thm: failbound general n}. 
\end{proof}

\bibliographystyle{plain}
\bibliography{ARRVref}

\end{document}